\documentclass[nohyperref]{article}

\usepackage{microtype}
\usepackage{graphicx}
\usepackage{subfigure}
\usepackage{booktabs} 
\usepackage[OT1]{fontenc} 
\usepackage{hyperref}
\usepackage{algorithm2e}
\RestyleAlgo{ruled}

\usepackage[accepted]{icml2023}

\usepackage{amsmath}
\usepackage{amssymb}
\usepackage{mathtools}
\usepackage{amsthm}

\usepackage[capitalize,noabbrev]{cleveref}
\usepackage{mathtools}       
\usepackage{amsmath}
\usepackage{amsthm}
\usepackage{accents}
\usepackage{amsmath,amssymb,amsthm}
\usepackage{multirow}

\usepackage[shortlabels]{enumitem}
\setlist[itemize]{topsep=0pt,}
\usepackage{mathrsfs}

\hypersetup{
colorlinks   = true, 
urlcolor     = blue, 
linkcolor    = blue, 
citecolor   = blue 
}

\usepackage{dsfont} 
\usepackage{xcolor} 

\usepackage{tabularx}
\usepackage{booktabs}
\usepackage[labelfont=bf,format=plain,justification=raggedright,singlelinecheck=false]{caption}

\def\tr{\mathop{\text{tr}}\kern.2ex}

\def \gam {{ \gamma }}
\def \indi {{ \mathds{1} }}

\def\R{{\mathbb R}}
\def\Rp{{\mathbb R _+}}
\def\Rn{{\mathbb R^n}}
\def\Rnp{{\mathbb R^n_+}}

\def\Rnpp{{\mathbb R^n_{++}}}

\def\P{{\mathbb P}}
\def\E{{\mathbb E}}
\def\B{{\mathbb B}}

\def\cC{{\mathcal{C}}}

\def\cF{{\mathcal{F}}}
\def\cH{{\mathcal{H}}}
\def\cG{{\mathcal{G}}}

\def\cX{{\mathcal{X}}}

\def\cD{{\mathcal{D}}}
\def\cP{{\mathcal{P}}}

\def\cN{{\mathcal{N}}}

\def\rg{{\rangle}}
\def\lg{{\langle}}

\newtheorem{Assumption}{Assumption}

\newtheorem{remark}{Remark}

\newtheorem{fact}{Fact}
\newtheorem{lemma}{Lemma}
\newtheorem{defn}{Definition}
\newtheorem{corollary}{Corollary}

\newlist{enumconditions}{enumerate}{1} 
\setlist[enumconditions]{label = \thelemma.\alph*}
\crefname{enumconditionsi}{Cond.}{Conditions}

\newlist{enumlmresult}{enumerate}{1} 
\setlist[enumlmresult]{label = \thelemma.\arabic*}
\crefname{enumlmresulti}{Part}{Parts}

\newlist{enumdef}{enumerate}{1} 
\setlist[enumdef]{label = \thedefn.\arabic*}
\crefname{enumdefi}{Cond.}{Cond.}

\crefname{Assumption}{Assumption}{Assumptions}
\renewcommand*{\theAssumption}{\arabic{Assumption}}
\crefname{defn}{Def.}{Defs.}

\newlist{enumthmresult}{enumerate}{1} 
\setlist[enumthmresult]{label = \thetheorem.\arabic*}
\crefname{enumthmresulti}{Part}{Parts}

\crefname{equation}{Eq.}{Eqs.}

\newlist{enumassumption}{enumerate}{1} 
\setlist[enumassumption]{label = \theAssumption.\arabic*}
\crefname{enumassumptioni}{Condition}{Conditions}

\crefname{appendix}{App.}{Apps}
\crefname{theorem}{Thm.}{Theorems}
\crefname{section}{Sec.}{Sections}

\newtheorem{theorem}{Theorem}
\renewcommand*{\thetheorem}{\arabic{theorem}}

\def \deltasti {{\delta^*_i}}
\def \deltast {{\delta^*}}
\def \deltagami {{\delta^\gamma_i}}
\def \deltagam {\delta^\gamma}

\def \I {{\mathrm{{I}}}}
\def \II {{\mathrm{{II}}}}
\def \III {{\mathrm{{III}}}}
\def \Ic {{I^{c}}}
\def \Icc {{I^{c}}}

\def \LAR  {{\small \textsc{{LAR}}}}
\def \REV {{\small \textsc{{REV}}}}
\def \REVst {{\small \textsc{{REV}}  }^* }
\def \REVgam {{\small \textsc{{REV}} }^ \gamma }

\def \NSW {{ \small \textsc{{NSW}}}}

\def \FPPE {{ \small \textsf{{FPPE}}}}
\def \oFPPE {{ \small \widehat{ \textsf{{FPPE}}}}}

\def \NSWgam {{ \small \textsc{{NSW}}}^\gamma}
\def \NSWst {{ \small \textsc{{NSW}}}^*}
\def \lin {\text{lin}}

\def \CR {{ \small \textsc{{CR}}}}
\def \CI {{ \small \textsc{{CI}}}}

\def \ust {{ u^* }}
\def \ugam { u^\gamma }
\def \ugami { u^\gamma_i }

\def \usti {{ u^*_i }}

\def \mubargam {{\bar \mu^\gamma}}
\def \mubargami {{\bar \mu^\gamma_i}}
\def \mugam {\mu^\gamma}

\def \must {{\mu^*}}
\def \mubarst {{\bar \mu^*}}
\def \musti {\mu^*_i}
\def \mubarsti {{\bar \mu^*_i}}

\def \vi {{v_i}}
\def \vithetau {{ v_i(\theta^\tau) }}
\def \vithe {{ v_i(\theta) }}

\def \vbar {{ \bar v }}

\def \fbar {{\bar f}}

\def \var {{ \operatorname {Var} }}
\def \cov {{ \operatorname {Cov} }}

\def \sumtau {{\sum_{\tau=1}^{t}}}

\def \sumiton {{ \sum_{i=1}^n }}

\def \pgam {{ p ^ {\gamma} }}

\def \xgam { x^{\gamma} }
\def \xgami {{ x^{\gamma}_i }}

\def \thetau {{ \theta^\tau }}

\def \eps{{ \epsilon }}

\def \pst {p^*}

\def \betast {\beta^*}

\def \betasti {{\beta^*_i}}

\def \betai {{\beta_i}}

\def \betagam {\beta^\gamma}

\def \betagami {\beta^\gamma_i}

\def \xitau {{x_i^\tau}}
\def \xst {x^*}
\def \xsti {{x^*_i}}

\def \inv {^{-1}}
\def \pinv {^{\dagger}}
\def \sq {^{2}}
\def \tp {^{\top}}
\def \st {^{*}}
\def \ot {^{\otimes 2}}

\def \Diag {{ \operatorname*{diag}}}
\def \Vec {{ \operatorname*{Vec}}}
\newcommand{\defeq}{\equiv}

\newcommand*\diff{\mathop{}\!\mathrm{d}}
\def \d {{\diff}}
\def \subdiff {{\text{diff}}}
\def \Ndiff {{N_{\text{diff}}}}
\def \Thetatie {{\Theta_{\text{tie}}}}

\def \toprob {{ \,\overset{{p}}{\to}\, }}

\def \tod {{ \,\overset{{\,d\,\,}}{\rightarrow}\, }}
\def \invtroot {t^{-1/2}}

\DeclareMathOperator*{\argmax}{arg\,max}

\DeclareMathOperator*{\esssup}{ess\,sup}
\def \relint {{ \operatorname*{relint}}}

\icmltitlerunning{Statistical Inference and A/B Testing for First-Price Pacing Equilibria}

\begin{document}

\twocolumn[
\icmltitle{Statistical Inference and A/B Testing for First-Price Pacing Equilibria}



\icmlsetsymbol{equal}{*}

\begin{icmlauthorlist}
\icmlauthor{Luofeng Liao}{sch}
\icmlauthor{Christian Kroer}{sch}
\end{icmlauthorlist}

\icmlaffiliation{sch}{IEOR, Columbia University}

\icmlcorrespondingauthor{Luofeng Liao}{ll3530@columbia.edu}

\icmlkeywords{First-price auction, pacing equilibria}

\vskip 0.3in
]



\printAffiliationsAndNotice{}  

\begin{abstract}
We initiate the study of statistical inference and A/B testing for first-price pacing equilibria (FPPE). 
The FPPE model captures the dynamics resulting from large-scale first-price auction markets where buyers use pacing-based budget management. Such markets arise in the context of internet advertising, where budgets are prevalent.

We propose a statistical framework for the FPPE model, in which a limit FPPE with a continuum of items models
the long-run steady-state behavior of the auction platform, and an observable FPPE consisting of a finite number of items provides the data to estimate primitives of the limit FPPE, such as revenue, Nash social welfare (a fair metric of efficiency), and other parameters of interest. We develop central limit theorems and asymptotically valid confidence intervals. Furthermore, we establish the asymptotic local minimax optimality of our estimators. We then show that the theory can be used for conducting statistically valid A/B testing on auction platforms. Numerical simulations verify our central limit theorems, and empirical coverage rates for our confidence intervals agree with our theory.

\end{abstract}
\section{Introduction}

A/B testing is a form of randomized controlled experiment, where each sample is assigned to one of two groups: the `A' group or `B' group, and a different treatment is applied to each group.
For example, say a social media site wants to test whether a new layout will increase user engagement. A subset of users are sampled, and each user in the subset is randomly assigned the current layout (group A), or the new layout (group B). Now, if we ignore network effects, then we can measure whether the new layout increases user engagement by checking whether engagement in group B is higher than in group A with statistical significance. As of 2017, large internet companies such as Google and Microsoft each conduct more than 10,000 A/B tests annually~\citep{kohavi2017surprising}.

However, now consider a setting where advertisers also bid on their ads being shown to users. If we randomly assign users to groups A and B, then we get interference because an advertiser's outcomes in group A affect their behavior when buying ads in group B. This is especially problematic if we are trying to measure something that directly pertains to ads (e.g. revenue changes, or user interest in the ads they are shown).
In practice, a popular solution to this issue is to create two separate markets, one for group A and one for group B. Then, each advertiser participates in both markets, with half of their budget assigned to each market, and those budgets are then treated as separate budget constraints. We will refer to this as \emph{budget splitting}.
Despite the practical popularity of budget splitting, its statistical properties are not well-understood.
A major obstacle to statistical inference with budget splitting is that we can no longer think of the mean user behavior as a sum of independent samples. Instead, we essentially have only two samples: a sample of a market under condition A, and a sample of a market under condition B. Thus, we need to understand when randomized assignment of users (which act as the supply of impressions) into separate markets can be used to make statistical inferences about market outcomes, given the effects of competition.

As stated at the beginning, we study this phenomenon in two important contexts: first-price pacing equilibria (FPPE), and Fisher markets.
We will focus the majority of our writing on FPPE, because FPPE model the advertising auction setting faced by large internet companies. All our results carry over to Fisher markets, except results on revenue, which are not meaningful in the standard Fisher market model.


In an FPPE, a set of buyers compete in a set of first-price auctions, and each buyer has a budget. 
This models how impressions are sold in practice, where first or second-price auction generalizations are used:
When a user shows up, an auction is run in order to determine which ads to show, before the page is returned to the user. This auction must run extremely fast. This is typically achieved by having each advertiser specify their target audience, their willingness-to-pay for an impression (or values per click, which are then multiplied by platform-supplied \emph{click-through-rate} estimates), and a budget ahead of time.
The control of the bids for individual impressions is then ceded to proxy bidders that are controlled by the ad platform.  
As a concrete example, to create an ad campaign on Meta Ads Manager, advertisers need to specify the following parameters: (1) the conversion location (how do you want people to reach out to you, via say website, apps, Messenger and so on), (2) optimization and delivery (target your ads to users with specific behavior patterns, such as those who are more likely to view the ad or click the ad link), (3) audience (age, gender, demographics, interests and behaviors), and (4) how much money do you want to spend (budget). 

Given the above parameters reported by the advertiser, the (algorithmic) proxy bidder supplied by the platform is then responsible for bidding in individual auctions so as to maximize advertiser utility, while respecting the budget constraint.
Two prevalent budget management methods are \emph{throttling} and \emph{pacing}. 
Throttling tries to enforce budget constraints by adaptively selecting which auctions the advertiser should participate in.
Pacing, on the other hand, modifies the advertiser's bids by applying a shading factor, referred to as a \emph{multiplicative pacing
multiplier}. Tuning the pacing multiplier changes the spending rate: the larger the pacing multiplier, the more aggressive the bids. 
The goal of the proxy bidder is to choose this pacing multiplier such that the advertiser exactly exhausts their budget (or alternatively use a multiplier of one in the case where their budget is not exhausted by using unmodified bids).
In this paper we focus on pacing-based budget management systems. 



In the case where each individual auction is a first-price auction, FPPE capture the outcomes of pacing-based budget-management systems.
\citet{conitzer2022pacing} introduced the FPPE notion, and showed that FPPE always exists and is unique. Moreover, FPPE enjoys lots of nice properties such as being revenue-maximizing among all budget-feasible pacing strategies, 
shill-proof (the platform does not benefit from adding fake bids under first-price auction mechanism) and
revenue-monotone (revenue weakly increases when adding bidders, items or budget).
Crucially for us, FPPE are fully characterized by a quasi-linear Eisenberg-Gale convex program~\citep{conitzer2022pacing,chen2007note}.

We remark that all the theory in the paper (CLT, inferential theory, and local asymptotic minimax theory) can be extended to Fisher market with quasilinear utility \citep{cole2017convex} given its equivalence to FPPE.

Given the above motivation, we study the question:
\begin{center}
\noindent\fbox{%

    \parbox{.4\textwidth}{%
    \emph{Suppose the auction platform operates at FPPE, i.e., at a market equilibrium. How can we quantify the variability in quantities of interest, and use this to perform A/B testing?}
    }%
}
\end{center}
Our contributions are as follows.

\paragraph{A statistical model for first-price pacing equilibrium and A/B testing in auction markets. }
We leverage the FPPE model of \citet{conitzer2022pacing} and the infinite-dimensional Fisher market model of \citet{gao2022infinite} in order to propose a statistical model for first-price auction markets.
In this model, we observe market equilibria formed with a finite number of items that are i.i.d.\ draws from some distribution, and aim to make inferences about several primitives of the limit market, such as revenue, Nash social welfare (a fair metric of efficiency), and other quantities of interest. 
More importantly, we lay the theoretical foundations for A/B testing in auction markets,
which is a difficult statistical problem because buyers interfere with each other through the supply and the budget constraints, the first-price auction allocation, and so on. 
With the presence of equilibrium effects, traditional statistical approaches which rely on the i.i.d.\ or the SUTVA (stable unit treatment value assumption, \citet{imbens2015causal}) assumption fail. The key lever we use to approach this problem is a convex program characterization of the first-price pacing equilibrium, called the Eisenberg-Gale (EG) program. With the EG program, the inference problem reduces to an $M$-estimation problem~\citep{shapiro2021lectures,van2000asymptotic} on a constrained non-smooth convex optimization problem.

\paragraph{Convergence and inference results for the limit market.}
The technical challenges for developing inferential theory for FPPE are two-fold:
(1) Nonsmoothness. The sample function in the convex program is non-differentiable on the constraint set almost surely as it involves the max operator (cf.\ \cref{eq:pop_deg}). 
Such nonsmoothness results from the fact that the allocation produced by first-price auction is highly nonsmooth w.r.t.\ buyer's pacing strategy.
(2) The parameter-on-boundary issue: the optimal population solution might be on the boundary of the constraint set.
Asymptotic convergence is established by showing that the EG convex program satisfies a set of regularity conditions from Theorem 3.3 by \citet{shapiro1989asymptotic}.
The hardest condition to verify is stochastic equicontinuity~(\cref{it:stoc_equic}), which we establish by leveraging empirical process theory~\citep{vaart1996weak,kosorok2008introduction}.
For constrained $M$-estimators, the asymptotic distribution might not be normal, causing challenges for inference. We discover sufficient conditions to ensure normality of the limit distributions. We also establish that the observed market is an optimal estimator of the limit market in the asymptotic local minimax sense \citep{van2000asymptotic,le2000asymptotics,duchi2021asymptotic}. Finally, we provide consistent variance estimators, whose consistency is proved by a uniform law-of-large-numbers over certain function classes.

\paragraph{Statistically-valid inference for A/B testing.}
Applying our theory, we develop an A/B testing design for item-side randomization that resembles practical A/B testing methodology.
In the proposed design treatment and control markets are formed, and buyer's budgets are split proportionally between them, while items are randomly assigned.
Then, based on the equilibrium outcomes, we construct estimators and confidence intervals that enable statistical inference. 
A recipe for applying our theory is presented in \cref{alg}.

\textbf{Notations.} 
For a measurable space $(\Theta, \diff \theta)$, we let $L^p$ (and $L^p_+$, resp.) denote the set of (nonnegative, resp.) $L^p$ functions on $\Theta$ w.r.t\ the integrating measure $\diff \theta $ for any $p\in [1, \infty]$ (including $p=\infty$). 
Given $x \in L^\infty$ and $v \in L^1$, we let $\langle v, x \rangle = \int_\Theta v(\theta) x(\theta) \diff \theta$.
We treat all functions that agree on all but a measure-zero set as the same.
Denote by $e_j$  the $j$-th unit vector.
For a sequence of random variables $\{X_n\}$, we say $X_n = O_p(1)$ if for any $\epsilon > 0$ there exists a finite $M_\epsilon$ and a finite $N_\epsilon$ such that $\P(|X_n| > M_\epsilon) < \epsilon$ for all $n\geq N_\epsilon$. We say $X_n = o_p(1)$ if $X_n$ converges to zero in probability.
We say $X_n = O_p(a_n)$ (resp.\ $o_p(a_n)$) if $X_n/a_n = O_p(1)$ (resp.\ $o_p(1)$).
The subscript $i$ is for indexing buyers and superscript $\tau$ is for items. 
Furthermore, we let $A\pinv $ be the Moore-Penrose pseudo inverse of a matrix $A$. Given vectors $a$ and $b$, let $a \small\odot b$ be the element-wise product.

In \cref{sec:related_work} we survey related works on A/B testing in two-sided markets, pacing equilibrum, $M$-estimation when the parameter is on the boundary, and statistical inference with equilibrium effects. 
\section{Statistical Model for First-Price Pacing Equilibrium}

Following~\citet{gao2022infinite,conitzer2022pacing}, we consider a single-slot auction market with $n$ buyers and a possibly continuous set of items $\Theta$ with {an} integrating measure $\diff \theta$.
For example, one could take $\Theta = [0,1]$ and $\diff \theta  = $ the Lebesgue measure on $[0,1]$.
Defining first price pacing equilibrium requires the following elements.
\begin{itemize}
    \item The \emph{budget} $b_i$ of buyer $i$. Let $b = (b_1,\dots, b_n)$. 
   
    \item The \emph{valuation} for buyer $i$ is a function $v_i \in L^1_+$, i.e., buyer $i$ has valuation $v_i(\theta)$ (value per unit supply) of item $\theta\in \Theta$. Let $v: \Theta \to \Rn$, $v(\theta) = [v_1(\theta),\dots, v_n(\theta)]$. We assume $\vbar = \max _i \sup_\theta  v_i(\theta)< \infty $.

    \item 
    The \emph{supplies} of items are given by a function  $ s \in L^\infty_+$, i.e., item $\theta\in \Theta$ has $s(\theta)$ unit of supply. 
    Without loss of generality, we assume a unit total supply $\int_\Theta s \diff \theta = 1$. 
    Given $g:\Theta \to \R$, we let $\E[g] = \int g(\theta)s(\theta)\diff \theta$ and $\var[g] = \E[g\sq] - (\E[g])\sq$.
\end{itemize}
For buyer $i$, an \emph{allocation} of items $x_i \in L^\infty_+$ gives a utility of 
$\langle v_i, sx_i \rangle $. Let $x = (x_1,\dots,x_n)$.
The \emph{prices} of items are modeled as $p\in L^1_+$; the price of item $\theta\in \Theta$ is $p(\theta)$.


\subsection{Definition and Interpretation of limit FPPE}
Central to the notion of an FPPE is the \emph{pacing multiplier}, which is a scalar $\beta_i \in [0,1]$ such that buyer $i$ bids their ``paced'' value $\beta_i \cdot v_i(\theta)$ on a given item $\theta$.
\begin{defn}[Limit FPPE]
    \label{def:limit_fppe}
    Given $(b,v,s)$, 
    a limit FPPE (denoted $\FPPE(b,v,s)$) is a tuple $(\beta, p, x_1, \dots, x_n) \in [0,1]^n \times L^1_+ \times  (L^\infty_+)^n$ such that 
    \begin{enumdef}[series = tobecont,itemjoin = \quad]
        \item (First-price) For all item $\theta \in \Theta$, $p(\theta) = \max_i \betai v_i(\theta)$. Moreover, $x_i(\theta) > 0$ implies $\betai \vithe =\max_k \beta_k v_k (\theta)$ for all $i$ and $\theta$.
        \label{it:def:first_price}
        \item (Supply and budget feasible)  For all $i$, $\int x_i(\theta) p(\theta) s(\theta)\diff \theta \leq b_i$. For all $\theta$, $\sumiton x_i (\theta) \leq 1$.  
        \label{it:def:supply_and_budget}
        \item (Market clearing) For all $i$, $\int x_i(\theta) p(\theta) s(\theta) \diff \theta < b_i$ implies $\betai = 1$. For all $\theta$, $p(\theta) > 0$ implies $ \sumiton x_i(\theta) = 1$.
        \label{it:def:rev_max}
    \end{enumdef}
\end{defn}
By \citet{conitzer2022pacing}, the limit FPPE exists and is unique. 
Next we unpack \cref{def:limit_fppe}. 
The $\beta_i \in [0,1]$ constraint is natural since a rational buyer should not bid more than their value.
\cref{it:def:first_price} captures the fact that FPPE is a model for first-price auctions, where pacing is used to manage budgets. The scalar $\beta_i$ controls the expenditure of buyer $i$, and the constraint ensures that the price is equal to the highest bid, and that only buyers tied for highest are allocated a non-zero amount.
\cref{it:def:supply_and_budget} ensures that the budget and supply constraints are satisfied.
\cref{it:def:rev_max} ensures that the solution satisfies \emph{no unnecessary pacing}, meaning that we should only scale down a buyer's bids in case their budget constraint is binding.
Secondly, it ensures that if a good is demanded by any buyers, then it must be fully allocated.

\begin{fact}[Buyer's satisfaction, Theorem 2 from \citet{conitzer2022pacing}]
    Let $(\beta, p , x) \in \FPPE(b,v,s)$.
    For all $i$, it holds $x_i \in \argmax_{x} \{ u_i(x) : 0 \leq x \leq 1, \lg p,  sx \rg \leq b_i\}$
    where the utility of a buyer is 
    \begin{align*}
        u_i(x_i)\defeq \lg v_i, sx_i \rg + (b_i - \lg p, sx_i \rg ) \;,
    \end{align*}
    where the first term is utility from the allocated items, the second term is the \emph{leftover budget}.
\end{fact}

This means FPPE is a competitive equilibrium. In the first-price auction context, each buyer's allocated items maximize their utility (item utility + leftover budget) among all budget-feasible allocations, given the price. 

From now on we use $(. )^*$ to denote limit FPPE quantities. 
Given an FPPE $(\betast, \pst, \xst)$, we define by
\begin{align}
    &\deltasti \defeq b_i - \lg \pst , s\xsti \rg \; , \quad \mubarsti \defeq \lg v_i, s\xsti \rg \;,
    \notag
    \\
    &\usti \defeq \mubarsti + \deltasti = u_i(\xsti)\;.
    \label{eq:def:must}
\end{align}
the \emph{leftover budget}, the \emph{item utility} and the \emph{total utility} of buyer $i$. 
Let $\deltast, \mubarst, \ust$ be the vectors that collect these quantities for all buyers.
It is well-known that $(\pst, \ust, \betast)$ are unique in equilibrium, but $(\xst, \deltast, \mubarst)$ might not be unique. Later we will see that for statistical inference we
need conditions to ensure uniqueness of $\xst$.
The following equations about limit FPPE \citep{gao2022infinite} are important.
\begin{align} \label{eq:compl_slackness}
    \usti = b_i / \betasti \; ,\quad  (\betasti-1)\deltasti = 0 \; .
\end{align}

We want to estimate the following quantities in the limit FPPE.
\textbf{(1) Revenue.} The revenue in the limit FPPE is 
$
        \REVst \defeq \int \pst(\theta) s(\theta)\diff \theta \; .
$
    It measures the profitability of the auction platform.
    When the platform operates at the limit FPPE, $\REVst$ is the maximum revenue the platform could extract from the buyers over the space of budget-feasible pacing strategies~\citet{conitzer2022pacing}.
\textbf{(2) Nash social welfare (NSW).} The (logarithm of) NSW is defined as 
$
        \NSWst \defeq \sumiton b_i \log \usti \;.
$
    The NSW at equilibrium measures total utility of the buyers and, when used as a summary metric of the efficiency of the auction platform, is able to promote fairness better than the utilitarian social welfare, that is, the sum of buyer utilities \citep{bertsimas2012efficiency,caragiannis2019unreasonable}. 
\textbf{(3) Individual utilities} at limit FPPE, $u^*_i$. 
\textbf{(4) Pacing multipliers} $\betasti$. 
    Pacing multiplier has a two-fold interpretation.
    First, through the equation $\betasti = b_i/ \usti$, it is 
    the ratio of budget and utility.
    Second, $\betast$ is the pacing policy employed by the buyers in first-price auctions, a quantity of natural interest.

As we will see next, counterparts of these quantities in the observed market are good estimators of the limit quantities.

\subsection{The Observed FPPE}

Let $\gamma = (\theta^1,\dots, \theta^t)$ be a sequence of observed items drawn from the distribution $s$ in an i.i.d.\ manner. Assume each item has the same supply of $\sigma \in \Rp$ units. Most of the time we take $\sigma = \frac1t$ to ensure total supply agrees with the limit market.
\begin{defn}[Observed FPPE, informal]
    \label{def:observed_fppe}
    Given $(b,v,\sigma, \gamma)$, the
    observed FPPE  $\oFPPE(b,v, \sigma , \gamma)$ contains tuples
    $(\beta,p, x_1,\dots, x_n) \in [0,1]^n \times \R^t_+ \times ([0, 1]^{t})^{n}$ 
    such that  
    the conditions in \cref{def:limit_fppe} hold with the following modifications:
    $
     \text{limit item space } \Theta \rightarrow \text{observed item set } \gamma$ and $\text{limit supply distribution $s(\cdot)$} \rightarrow \text{weighted observed item distribution } \sigma \sumtau \delta_{\thetau}(\cdot)$
    where $\delta_\theta(\cdot)$ is a point mass on $\theta$.
\end{defn} 
A formal definition and further properties of FPPE can be found in \cref{sec:fppe_properties}.
Here $\xitau \in [0,1]$ is the fractional allocation of item $\thetau$ to buyer $i$.
The mechanism of forming the observed FPPE is exactly the same as the limit FPPE in \cref{def:limit_fppe}, except now the price $p$, the supply $s$ and the allocation $x_i$ reduce to vectors in $\R^t$ as opposed to functions. 

To emphasize dependence on the item sequence $\gamma$, we use $(.)^\gamma$ to denote equilibrium quantities in $\oFPPE(b,v,t\inv, \gamma)$. 
We let $(\betagam,\pgam,\xgam)$ be an observed FPPE with $\xgam = (\xgam_1,\dots,\xgam_n)$. 
The leftover budget $\delta^\gam_i \defeq b_i- \sigma \lg \pgam, \xgami\rg$, item utility $\mubargami \defeq \sigma \lg v_i, \xgami \rg$ and total utility $\ugami \defeq \delta^\gam_i + \mubargami $ are defined similarly. Let $\deltagam, \mubargam, \ugam$ be the vectors that collect these quantities for all buyers. The observed revenue is $\REVgam \defeq \sigma \sumtau \pgam(\thetau)$, and NSW is $\NSWgam \defeq \sumiton b_i \log \ugami$.

Having observed an FPPE with a finite number of items, our goal is to estimate the quantities of interest in the limit FPPE.

\subsection{Convex Programs for FPPE}
Before we present the main statistical theories for FPPE, we review convex program characterization of FPPE, which are at the core of the paper. 
These convex characterization results reduce the FPPE inference problem to the one on a constrained nonsmooth convex program. 
\label{sec:convex_program}

The limit FPPE pacing multiplier $\beta$ can be recovered through
the population dual Eisenberg-Gale (EG) program
\begin{align}\hspace{-.2cm}
    \label{eq:pop_deg}
    \min_{B}
     H(\beta)\defeq\E[F(\theta, \beta)]
    \tag{\small P-DualEG}
\end{align}
where $B \defeq (0, 1]^{n}$, $F \defeq f + \Psi$, $$f(\theta,\beta) \defeq \max_{i\in[n]} \beta_i v_i(\theta) \;, \; \Psi(\beta) \defeq -  \sumiton b_i \log \beta_i\;.$$
It is known that 
the FPPE $\betast$ is the unique solution to \cref{eq:pop_deg}, and any FPPE $(\xst, \ust, \deltast)$ belongs to the set of optimal solutions to the population primal EG program (to be presented in \cref{sec:fppe_properties}).

The observed FPPE has a similar convex program characterization. By taking $\sigma = \frac1t$ and replacing $s$ with the empirical measure $\frac1t \sumtau \delta_\thetau$ in \cref{eq:pop_deg}, it can be shown that the observed FPPE pacing multiplier solves the sample dual EG program
\begin{align}
    \min_{\beta \in B}
     H_t(\beta) \defeq \frac1t \sumtau  F(\thetau, \beta)
     \;.
    \tag{\small S-DualEG}
    \label{eq:sample_deg}
\end{align}
To develop inferential theory for FPPE, we study the concentration of the \emph{dual} EG programs. The study of the convergence ``$ \oFPPE (b,v,t\inv, \gamma)\!\!\implies\!\!\FPPE(b,v,s)$" reduces to the one of 
\begin{align*}
    ``\min_{\beta \in B}  H_t(\beta)
    \quad \!\! \implies \!\! \quad 
    \min_{\beta \in B}  H(\beta) 
    "\; .
\end{align*}
As mentioned previously, the difficulty of analyzing the above convex programs lies in the nonsmoothness of the sample function $F$ and that the population optimum could lie on the boundary of $B$. We define the set of constraints that are active/inactive at $\betast$ by
\begin{align}
    I \defeq \{i : \betasti = 1\} \;,\;\;\Ic \defeq  \{ i: \betasti < 1\}\;.
\end{align}

\section{Statistical Inference}
\label{sec:stat inf}

For statistical inference we need the limit market to behave smoothly around the optimal pacing multipliers $\betast$.
To that end, we make the following assumption.
\begin{Assumption}[\textsf{\scriptsize{SMO}}]
    \label{as:smoothness}
Assume the map $\beta \mapsto \E_s[\max_i \betai \vithe ]$ is $C^2$ in a neighborhood of the limit FPPE pacing multiplier $\betast$.
\end{Assumption}
For a given $\beta \in \Rnp$, the quantity $\max_i \betai \vithe$ is the price of item $\theta$ in the first-price auction. The assumption requires that the revenue, $\E_s[\max_i\betai\vithe]$, when viewed as a function of $\beta$, changes smoothly around $\betast$. The assumption implies that $H$, defined in \cref{eq:pop_deg}, is also $C^2$ at $\betast$.

\cref{as:smoothness} implies a number of nice regularity conditions. One is that the set of items that are tied at the limit FPPE is $s$-measure zero. The set of tied items is
$
    \Thetatie \defeq \{\theta \in \Theta: \betasti \vithe = \beta^*_k v_k(\theta)\text{ for some $i\neq k$} \} \;.
$
\begin{lemma}
    \label{lm:hessian_implications}
    Under \nameref{as:smoothness}, the set $\Thetatie$ is $s$-zero measure (up to a measure-zero set), and the equilibrium allocation $\xst$, the leftover budget $\deltast$ and the item utility $\mubarst$ are all unique.  Moreover, there is a neighborhood $\Ndiff$ of $\betast$ such that each pacing strategy in this neighborhood results in no tie.
    Proof in \cref{sec:fppe_properties}.
\end{lemma}

\nameref{as:smoothness} is a joint assumption on value functions $v$ and the supply function $s$. Lower level conditions on $v$ and $s$ that imply \nameref{as:smoothness} were derived by \citet{liao2022stat}. For example, if the distribution of the values $v_i$ is smooth, then \nameref{as:smoothness} holds. If we impose functional structure on $v$, such as $\Theta=[0,1]$ and $v_i = a_i \theta + b_i$ with $\E[v_i]=1$, $b_i$'s distinct, and $s$ is uniform, then \nameref{as:smoothness} also holds. 
If the gap between the highest and the second-highest bid is large for most items in the limit FPPE, \nameref{as:smoothness} also holds. For a precise statement, we refer readers to Theorem 7 from \citet{liao2022stat}.

In order to state our main CLT results, we define 
\begin{align}
    \label{def:Omega}
    \must(\theta) \defeq \xst(\theta)\odot v(\theta)\; , \;\;  
    \cH\defeq \nabla\sq H (\betast) \;.
\end{align}
Under \nameref{as:smoothness}, $\must(\cdot)$ is unique and well-defined. Clearly $\mubarst = \E[\must(\theta)]$.

In the unconstrained case, classical $M$-estimation theory says that, under regularity conditions, an $M$-estimator is asymptotically normal with covariance matrix $\cH\inv  \var(\text{gradient}) \cH\inv$ \citep[Chap.\ 5]{van2000asymptotic}. However, in the case of FPPE which is characterized by a constrained convex problem, the Hessian matrix needs to be adjusted to take into account the geometry of the constraint set $B=(0,1]^n$ at the optimum $\betast$.
We let $\cP \defeq \Diag( \indi(i \in \Ic \,) )$ be an ``indicator matrix'' of buyers whose $\betasti < 1$, and define the projected Hessian 
\begin{align}
    \label{def:cHB}
    \cH_B \defeq \cP \cH \cP \;.
\end{align}
It will be shown that the asymptotic variance of $\betagam$ is $\cH_B\pinv \var(\text{gradient}) \cH_B\pinv$.

\begin{Assumption}[\textsf{\scriptsize{SCS}}] 
    \label{as:constraint_qualification} Strict complementary slackness holds:
     $\betasti = 1$ implies  $\delta^*_i  > 0$.
\end{Assumption}
\nameref{as:constraint_qualification} can be viewed as a non-degeneracy condition from a convex programming perspective, since $\delta_i$ corresponds to a Lagrange multiplier on $\beta_i \leq 1$.
From a market perpective, \nameref{as:constraint_qualification} requires that if a buyer's bids are not paced ($\betasti = 1$), then the leftover budget $\delta^* _i$ must be strictly positive.
This can again be seen as a market-based non-degeneracy condition: if $\deltasti=0$ then the budget constraint of buyer $i$ is binding, yet $\betasti=1$ would imply that they have no use for additional budget.
If \nameref{as:constraint_qualification} fails, one could slightly increase the budgets of buyers for which \nameref{as:constraint_qualification} fails, i.e., those who do not pace yet have exactly zero leftover budget, and obtain a market instance with the same equilibrium, but where \nameref{as:constraint_qualification} holds.

From a technical viewpoint, \nameref{as:constraint_qualification} is a stronger form of the first-order optimality condition.  
Note $\nabla H(\betast) = -\deltast$ (cf.\ \cref{sec:fppe_properties}). 
The usual first-order optimality condition is 
\begin{align}
    - \nabla H(\betast) \in \cN_B(\betast) \;,
    \label{eq:usual_foc}
\end{align}
where $\cN_B(\beta) = \prod_{i=1}^n J_i(\beta)$ is the normal cone with $J_i (\beta)= [0,\infty)$
if $\betai = 1$ and $J_i (\beta)= \{0\}$ if $\betai < 1$ for $\beta \in \Rnpp$.
Then \cref{eq:usual_foc} translates to the condition that $\betasti = 1$ implies $\deltasti \geq 0$. 
On the other hand, when written in the form that resembles optimality condition, 
\nameref{as:constraint_qualification} is equivalent to $$- \nabla H(\betast) \in \relint ( \cN_B(\betast) ) \; .
\footnote{
    The relative interior of a set is $\relint(S)\defeq\{x \in S:$ there exists $\epsilon>0$ such that $N_\epsilon(x) \cap \operatorname{aff}(S) \subseteq S \}$ where $\operatorname{aff}(S)$ is the affine hull of $S$, and $N_\epsilon(x)$ is a ball of radius $\epsilon$ centered on $x$. 
}$$
Given that $ \relint ( \cN_B(\betast) ) \subset \cN_B(\betast) $, \nameref{as:constraint_qualification} is obviously a stronger form of first-order condition. The \nameref{as:constraint_qualification} condition is commonly seen in the study of statistical properties of constrained $M$-estimators (\citet[Assumption B]{duchi2021asymptotic} and \citet{shapiro1989asymptotic}). In the proof of \cref{thm:clt}, \nameref{as:constraint_qualification} forces the critical cone to reduce to a hyperplane and thus ensures asymptotic normality of the estimates.
Without \nameref{as:constraint_qualification}, the asymptotic distribution of $\betagam$ could be non-normal.

\subsection{Central Limit Theorems}
We now show that the observed pacing multipliers $\betagam$ and the observed revenue $\REV^\gamma$ converge to the limit market quantities in probability, and satisfy central limit theorems.
Define the \emph{influence functions} 
\begin{equation} \label{eq:def_inf_func}
    \begin{aligned}
    D_\beta(\theta) & \defeq - (\cH_B)^\dagger (\must(\theta) - \mubarst) ,
    \\
    D_\REV(\theta) & \defeq \pst(\theta) - \REV^* + (\mubarst) \tp D_\beta(\theta) .
    \end{aligned}
\end{equation}
Recall $\must$ is defined in \cref{def:Omega}, $\cH_B$ in \cref{def:cHB}. Clearly $\E[D_\beta] = 0$ and $\E[D_\REV] = 0$.

\begin{theorem} 
    \label{thm:clt}
    It holds that $\betagam \toprob \betast$ and $\REV^\gamma \toprob \REV^*$. Furthermore, if
    \nameref{as:constraint_qualification} and \nameref{as:smoothness} hold, then 
    \begin{align*}
        &\sqrt{t}(\betagam - \betast) =  \invtroot \sumtau D_\beta (\thetau) + o_p(1) \; ,
        \\
        &\sqrt{t} (\REV^\gam - \REV^*) = \invtroot \sumtau D_\REV(\thetau) + o_p(1)  \;.
    \end{align*}
    Consequently, $ \sqrt{t}(\betagam - \betast )$ and $\sqrt {t} (\REV^\gamma - \REV^*) $
    are asymptotically normal with means zero and variances $\Sigma_\beta\defeq \E[D_\beta \ot ] =  (\cH_B)\pinv \var(\must) (\cH_B)\pinv$ and $\sigma^2_\REV \defeq \E[D_\REV (\theta)^2]$. 
    Proof in \cref{sec:proof:thm:clt}.
\end{theorem}

The functions $D_\beta$ and $D_\REV$ are called the {influence functions} of the estimates $\betagam$ and $\REV^\gam$ because they measure the change in the estimates caused by adding a new item to the market (asymptotically).

\cref{thm:clt} implies fast convergence rate of $\betagami$ for $i$ whose constraint is tight in the limit market. 
To see this, we 
suppose wlog.\ that $I = [k]$, i.e. the  first $k$ buyers are the ones with $\beta_i=1$.
Then the pseudo-inverse of projected Hessian $(\cH_B )\pinv = \Diag(0_{k\times k} , (\cH_{\Ic \Ic})\inv )$ where $\cH_{\Ic \Ic}$ is the lower right $(n-k)\times(n-k)$ block of $\cH $. 
Consequently, entries of $\Sigma_u$ (resp.\ $\Sigma_\beta$)  are zeros except those on the lower right $(n-k)\times(n-k)$ blocks $(\Sigma_u)_{\Ic\Ic}$ (resp.\ $(\Sigma_\beta)_{\Ic\Ic}\,$).
This result shows that the constraint set $B$ ``improves'' the covariance by zeroing out the entries corresponding to the active constraints $I$.
Consequently, 
$ \sqrt{t} (\betagami - \betasti)$ and $\sqrt t (\ugami - \usti)$ are of order $o_p(1)$ for $i \in I$, and thus converging faster than the usual $O_p(1)$ rate.
The fast rate phenomenon is empirically investigated in \cref{sec:exp_fast_rate}. 

By \nameref{as:constraint_qualification} we have $I = \{ i: \deltasti > 0\}$, i.e., $I$ is the set of buyers with positive leftover budgets, and $\Ic\, = \{ i: \deltasti = 0\} $, i.e., $\Ic\,$ is the set of buyers who exhaust their budgets. \footnote{Without \nameref{as:constraint_qualification}, it only holds $ \Ic\, \subset \{i:\deltasti = 0\}$ and $\{i:\deltasti > 0 \}\subset I$ by complementary slackness \cref{eq:compl_slackness}.}
In the context of first-price auctions,
the fast rate $o_p(t^{-1/2})$ implies that the platform can identify buyers that are unpaced in the limit FPPE even when the market size is small.

The proof of \cref{thm:clt} proceeds by showing that FPPE satisfy a set of regularity conditions that are sufficient for asymptotic normality~\citep[Theorem 3.3]{shapiro1989asymptotic}; the conditions are stated in \cref{thm:the_shapiro_thm} in the appendix. Maybe the hardest condition to verify is the so called stochastic equicontinuity condition (\cref{it:stoc_equic}), which we establish with tools from the empirical process literature. 
In particular, we show that the function class whose functions map an item to the first-price auction allocation of items, is VC-subgraph, which implies stochastic equicontinuity. 
\nameref{as:constraint_qualification} is used to ensure normality of the limit distribution.

Finally, we remark that the CLT result for revenue holds true 
even if $I = \emptyset$.
If $\betagami < 1$ for all $i$, then all buyers' budgets are exhausted in the observed FPPE, and so we have
$
    \REVgam = \REVst = \sumiton b_i$ if $I = \emptyset \;.$
By the convergence $\betagam \toprob \betast$, we know that 
$\REVgam = \REVst$ with high probability for all large $t$ if $I=\emptyset$.
In that case, it must be that the asymptotic variance of revenue equals zero.
Our result covers this case because one can show $\sigma\sq_\REV = 0$ using 
Euler's identify for homogenous functions and that $\cH \betast = \must$ if $I = \emptyset$; see \cref{lm:p_equal_muHinvmu}.

By applying the delta method based on \cref{thm:clt}, we can derive a CLT for individual utilities $\ugam$, leftover budget $\deltagam$ and Nash social welfare $\NSW^\gam$ since they are smooth functions of $\betagam$; see \cref{lm:from_eg_to_fppe}.
\begin{corollary}
    \label{cor:clt_u_and_nsw}
    Under the same conditions as \cref{thm:clt},
       $\sqrt{t} (\ugam - \ust) $, $\sqrt t(\deltagam - \deltast)$ and $\sqrt{t}(\NSW^\gam - \NSW^*) $
       are asymptotically normal with means zero and 
(co)variances defined in \cref{sec:full:cor:clt_u_and_nsw}.
\end{corollary}

\subsection{Asymptotic Local Minimax Optimality}
Given the asymptotic normality of observed FPPE, it is desirable to understand the best possible statistical procedure for estimating the limit FPPE.
One way to discuss the optimality is to measure the difficulty of estimating the limit FPPE when the supply distribution varies over small neighborhoods of the true supply $s$, asymptotically. When an estimator achieves the best worst-case risk over these small neighborhoods, we say it is asymptotically locally minimax optimal. For general references, see \citet{vaart1996weak,le2000asymptotics}. More recently \citet[Sec.\ 3.2]{duchi2021asymptotic} develop asymptotic local minimax theory for constrained convex optimization, and we rely on their results.

Given the central limit results for $\beta, u, \NSW$ and $\REV$, we will show that the observed FPPE estimates are optimal in a asymptotic local minimax sense.
To make this precise we introduce a few more notations to parametrize neighborhoods of the supply $s$.
Let $g\in G_d = \{g:\Theta \to \R^d: \E[g]=0, \E[\|g\|\sq]<\infty\}$ be a direction along which we wish to perturb the supply $s$.
Given a vector $\alpha \in \R^d $ signifying the magnitude of perturbation,
we want to scale the original supply of item $\theta$ by  $\exp(\alpha\tp g(\theta))$ and then obtain a perturbed supply distribution by appropriate normalization.
To do this we define the perturbed supply by  
\footnote{
    In \citet{duchi2021asymptotic} they allow more general classes of perturbations, we specialize their results for our purposes.
}
\begin{align}
    \label{eq:perturbed_is_roughly_expo}
    s_{\alpha,g} (\theta) \defeq C \inv [{1 + \alpha\tp g(\theta)}] s(\theta) 
\end{align}
with a normalizing constant $C = 1 + \int \alpha\tp g(\theta) s(\theta)\diff\theta$.  
As $\alpha\to 0$, the perturbed supply $s_{\alpha,g}$ effectively approximates
$
    s_{\alpha,g} (\theta) \propto \exp(\alpha\tp g(\theta)) s(\theta) 
.$

\paragraph{Asymptotic local minimax optimality for $\beta$.}
We first focus on estimation of pacing multipliers.\
For a given perturbation ${(\alpha,g)}$, we let $\betast_{\alpha,g}$, $\pst_{\alpha,g}$ and $\REVst_{\alpha,g}$ be the limit FPPE pacing multiplier, price and revenue under supply distribution $s_{\alpha,g}$.
Clearly $\betast = \betast_{0,g}$ for any $g$ and similarly for $\pst_{\alpha,g}$ and $\REV_{\alpha, g}$.
Let $L:\Rn \to \R$ be any symmetric quasi-convex loss. 
\footnote{A function is quasi-convex if its sublevel sets are convex.}
In asymptotic local minimax theory we are interested in the local asymptotic risk: given a sequence of estimators $\{ \hat \beta_t: \Theta^t \to \Rn \}_t$, $\LAR_\beta (\{\hat \beta_t\}_t )  \defeq $
\begin{align*}
    {\sup_{ g\in G_d, d\in \mathbb{N}}
    \lim_{c\to \infty}
    \liminf_{t \to \infty}
    \sup_{\|\alpha\|_2\leq \frac{c}{\sqrt t}}
    \E_{s_{\alpha,g}^{\otimes t}}[L(\sqrt{t} (\hat \beta_t - \betast_{\alpha,g} ))] } \;.
\end{align*}
If we ignore the limits and consider a fixed $t$, then $\LAR_\beta$ roughly measures the 
worst-case risk for the estimators $\{\hat \beta_t\}$.
Note that $\alpha$ is a $d$-dimensional vector, and thus the shrinking norm-balls depend on the choice of $d$, and the expectation is taken w.r.t.\ the $t$-fold product of the perturbed supply.
As an immediate application of Theorem 1 from \citet{duchi2021asymptotic}, it holds that 
\begin{align*}
    \inf_{ \{\hat \beta_t\}_t }  {\small \textsc{{LAR}}}_\beta (\{\hat \beta_t\}_t )  \geq \E[L(\cN( 0, (\cH_B)\pinv \var(\must) (\cH_B)\pinv))] \; . 
\end{align*}
where the expectation is taken w.r.t.\ a normal specified above.
Moreover, the lower bound is achieved by the observed FPPE pacing multipliers $\betagam$ according to the CLT result in \cref{thm:clt}.

\paragraph{Asymptotic local minimax optimality for revenue estimation.}
More interesting is the asymptotic minimax result of revenue estimation.
Given a symmetric quasi-convex loss $L:\R\to\R$, we define the local asymptotic risk for any procedure $\{\hat r _t : \Theta^t \to \R\}$ that aims to estimate the revenue: $\LAR_\REV ( \{\hat r_t \}) \defeq$
\begin{align*} \hspace{-.2cm}
    \sup_{ g\in G_d, d\in \mathbb{N}}
    \lim_{c\to \infty}
    \liminf_{t \to \infty}
    \sup_{\|\alpha\|_2\leq \frac{c}{\sqrt t}}
    \E_{s_{\alpha,g} ^{\otimes t}}[L(\sqrt{t} (\hat r_t  - \REVst_{\alpha,g} ))] \;. 
\end{align*} 
\begin{theorem}[Asymptotic local minimaxity for revenue]
    \label{thm:rev_localopt}
   If \nameref{as:smoothness} and \nameref{as:constraint_qualification} hold, then 
    \begin{align*}
       \inf_{ \{\hat r_t\}} \LAR_\REV ( \{\hat r_t \}) \geq \E[L(\cN(0, \sigma\sq_\REV))] \;.
    \end{align*} 
\end{theorem}

Proof in \cref{sec:proof:rev_local_asym_risk}. In the proof we calculate the derivative 
of revenue w.r.t.\ $\alpha$, which in turns uses 
a perturbation result for constrained convex program from \citet{duchi2021asymptotic,shapiro1989asymptotic}.
Again, the lower bound is achieved by the observed FPPE revenue $\REVgam$ according to the CLT result in \cref{thm:clt}.
Similar optimality statements can be made for $u$ and $\NSW$ by finding the corresponding derivative expressions.
\subsection{Inference}
\label{sec:inference}

In order to perform inference, we need to construct estimators for the influence functions \cref{eq:def_inf_func}. We show how each component can be estimated by the observed FPPE.

Given a sequence of smoothing parameters $\varepsilon_t = o(1)$, we estimate $\cP$ by 
$
    \hat \cP \defeq \Diag(\indi(\betagami < 1 - \varepsilon_t )) \;.
$
For the same sequence $\varepsilon_t$, we introduce a numerical difference estimator $\hat \cH$ for the Hessian matrix $\cH$, whose $(i,j)$-th entry is $  \hat{\cH}_{i j}\defeq$
$
  [H_t(\betagam_{++})-H_t(\betagam_{+-})
     -H_t(\betagam_{-+}) 
    +H_t(\betagam_{--})] / 4 \varepsilon_t^2 
$
with 
$\betagam_{\pm \pm} \defeq  \betagam \pm e_i \varepsilon_t \pm e_j \varepsilon_t$,
and $H_t$ is defined in \cref{eq:sample_deg}. Then $\hat \cH_B = \hat \cP \hat \cH \hat \cP$ is a natural estimator of $\cH_B$. 
Also, $\xst, \pst$ will be estimated by $\xgam, \pgam$.
Let $\xgami(\thetau) \defeq (\xgami)^\tau \in [0,1]$ be the proportion of $\thetau$ allocated to buyer $i$ and $\pgam(\thetau) \defeq (\pgam)^\tau$ be the price. 

Mirroring the definitions in \cref{eq:def_inf_func}, we  define the following influence function estimators
\begin{align*}
    \hat D_\beta^\tau & \defeq - (\hat \cH_B)\pinv (\xgam(\thetau) \odot v(\thetau) - \mubargam) \;,
    \\ 
    \hat D _\REV ^\tau &\defeq \pgam (\thetau) - \REV^\gamma + (\mubargam) \tp\hat D_\beta^\tau \;.
\end{align*}
Given that the asymptotic variance of $\betagam$ (resp.\ $\REVgam$) is 
$\E[D_\beta\ot]$ (resp.\ $\E[D_\REV\sq]$), plug-in estimators of the (co)variance are naturally
\begin{align}
    \hat \Sigma_\beta \defeq \frac{1}{t} \sumtau \hat D_\beta^\tau  (\hat D_\beta^\tau )\tp \; , \;\;
    \hat \sigma \sq_\REV \defeq \frac{1}{t} \sumtau (\hat D_\REV ^\tau)\sq \; .
    \label{eq:plugin variance}
\end{align}
For any $\alpha \in (0,1)$, the $(1-\alpha)$-confidence region/interval for $\betast$ and $\REVst$ are 
\begin{align}
   & \CR_\beta\defeq \betagam + (\chi_{n, \alpha} / \sqrt{t})\hat \Sigma_\beta^{1/2} \B \; ,\label{eq:beta CI}
    \\ 
   & \CI_\REV \defeq[\REVgam\pm z_{\alpha/2}  \hat \sigma_{\REV}/\sqrt t] \; .
   \label{eq:rev_variance}
\end{align} 
where $\chi_{n,\alpha}$ is the $(1-\alpha)$-th quantile of a chi-square distribution with degree $n$, $\B$ is the unit ball in $\Rn$, and $z_{\alpha/2}$ is the $(1-\frac{\alpha}{2})$-th quantile of a standard normal distribution.
The coverage rate of $\CI_\REV$ is empirically verified in \cref{sec:exp_rev_CI}.
\begin{theorem}
    \label{thm:variance_estimation}
    Under the conditions of \cref{thm:clt}, 
    let $\varepsilon_t \sqrt{t}\to \infty$ and $\varepsilon_t \downarrow 0$. 
    Then $\hat \Sigma_\beta \toprob \Sigma_\beta$ and $\hat \sigma_\REV \sq \toprob \sigma_\REV \sq$. 
    Consequently, for any $\alpha \in (0,1)$,
    $\P (\betast \in \CR_\beta ) \to 1-\alpha$ and $\P(\REVst \in \CI_\REV ) \to 1-\alpha$. 
    Proof in \cref{sec:proof:thm:variance_estimation}.
\end{theorem}
The theorem suggests choosing smoothing parameter $\varepsilon_t = t^{- d}$ for $0 < d < \frac12$; see \cref{sec:exp_hessain} for a numerical study on how $d$ affects Hessian estimation.
Variance estimators for $u$, $\delta$ and $\NSW$ can be constructed similarly. 

\section{A/B Testing for First-Price Auction Platforms}
\label{sec:ab testing}

Consider an auction market with $n$ buyers with a continuum of items $\Theta$ with supply function $s$.
To model treatment application we introduce the \emph{potential value functions}
$$ \hspace{-.2cm} v(0) \defeq (v_1(0,\cdot),\dots, v_n(0,\cdot)),\;v(1)\defeq (v_1(1,\cdot),\dots, v_n(1,\cdot)).$$
If item $\theta$ is exposed to treatment $w \in \{0,1\}$, then its value to buyer $i$ will be $v_i(w,\theta)$. 

Suppose we are interested in estimating the change in the auction market when treatment 1 is deployed to the entire item set $\Theta$. 
In this section we describe how to do this using A/B testing, specifically for estimating the treatment effect on revenue. 
We discuss other quantities like Nash social welfare in \cref{sec:abtesting_more}.
Formally, we wish to look at the difference in revenues between the markets 
$$
    \FPPE(b, v(0), s) \text{ and }  \FPPE(b,v(1), s),
$$
where $\FPPE(b, v(0), s)$ is the market with treatment 1, and $\FPPE(b,v(1), s)$ is the one with treatment 0. 
The treatment effects on revenue is defined as $$\tau_{\REV} \defeq \REVst(1) - \REVst(0)\;,$$
where $\REVst(w)$ is revenue in the equilibrium $\FPPE(b,v(w), s)$.

We will refer to the experiment design as \emph{budget splitting with item randomization}.
\textbf{Step 1. Budget splitting.}
    We replicate buyers by splitting their budgets and form two markets with the same set of buyers. 
    For each buyer $i$ we allocate $\pi b_i$ of their budget to the market with treatment $w=1$, and the remaining budget, $(1-\pi)b_i$, to the market with treatment $w=0$.
\textbf{Step 2. Item randomization.}
Let $(\theta^1,\theta^2,\dots)$ be i.i.d.\ draws from the supply distribution $s$.
For each sampled item, it is applied treatment $1$ with probability $\pi$ and treatment $0$ with probability $1-\pi$.  
The total A/B testing horizon is $t$. When the end of horizon is reached, two observed FPPEs are formed. Assume each item has a supply of $\pi/t_1$ in the 1-treated market and $(1-\pi)/t_0$ in the 0-treated market. The $1/t_1$ is the scaling required for our CLTs and the $\pi$ factor ensures the budget-supply ratio agrees with the limit market; see \cref{lm:scale_invariance_fppe} regarding scale-invariance of FPPE.

Let $t_0$ be the number of $0$-treated items, and $t_1$ be the number of $1$-treated items. Conditional on the total number of items $t=t_1 + t_0$, the random variable $t_1$ is a binomial random variable with mean $\pi t$. Let $\gamma(0) = (\theta^{1,1},\dots, \theta^{1,t_1})$ be the set of $0$-treated items, and similarly $\gamma(1) = (\theta^{0,1},\dots, \theta^{0,t_0})$.
The total item set $\gamma = \gamma(0) \cup \gamma(1)$. 
Compactly, the observables in the described A/B testing experiment are 
$$ \hspace{-.1cm}
\oFPPE \big(\pi b, v(1), \tfrac{\pi}{t_1}, \gamma(1)  \big),\; 
    \oFPPE \big( (1-\pi ) b,  v(0), \tfrac{1-\pi}{t_0}, \gamma(0)\big ), $$
both defined in \cref{def:observed_fppe}.
Let $\REVgam(w)$ denote the observed revenue in the $w$-treated market.
The estimator of revenue treatment effect is 
$$
     \hat \tau_{\REV} \defeq \REVgam(1) - \REVgam(0).$$

For fixed $(b, s)$, the variance $\sigma\sq_\REV$ in \cref{thm:clt} is a functional of the value functions. We will use $\sigma\sq_\REV(w)$ to represent the revenue variance in the equilibrium $\FPPE(b,v(w), s)$.
Each variance can be estimated using \cref{eq:plugin variance}.

\begin{theorem}[Revenue treatment effects CLT]
    \label{thm:clt_ab_testing}
    Suppose \nameref{as:smoothness} and \nameref{as:constraint_qualification} hold in the limit markets $\FPPE(b,v(1),s)$ and $\FPPE(b,v(0),s)$. Then 
 $\sqrt{t} (\hat \tau_\REV - \tau_\REV) \tod\cN\allowbreak \big(0, \frac{ \sigma\sq_\REV(1)}{\pi} \allowbreak + \allowbreak\frac{ \sigma\sq_\REV(0)}{(1-\pi)} \big) .$
Proof in \cref{sec:proof:thm:clt_ab_testing}.
\end{theorem}

Based on the theorem, an A/B testing procedure is the following.
Compute the revenue variance as \cref{eq:plugin variance} for each market, obtaining $ \hat \sigma\sq_\REV(1)$ and $ \hat \sigma\sq_\REV(0)$, and form the confidence interval
\begin{align}
    \hat \tau_\REV \pm z_{\alpha/2} \bigg( \frac{ \hat \sigma\sq_\REV(1)}{\pi} \allowbreak + \allowbreak\frac{ \hat \sigma\sq_\REV(0)}{(1-\pi)}  \bigg)
    \label{eq:ab_rev_ci}
\end{align}
If zero is on the left (resp.\ right) of the CI \cref{eq:ab_rev_ci}, we conclude that the new feature increases (resp.\ decreases) revenue with $(1-\alpha)\times 100\%$ confidence. 
If zero is inside the interval, the effect of new feature is undecided.
See \cref{sec:exp_abtest_CI} for a numerical study verifying the validity of this procedure.
\cref{alg} presents a step-by-step procedure for using the above results.



\section{Experiment}
In 
\cref{sec:exp_details}
we conduct simulations to investigate asymptotic normality for $i \notin I$ and fast convergence rate for $i\in I$,
the effect of smoothing parameter on Hessian estimation,  
the coverage rate of the CI for revenue, and
the coverage rate of the treatment effect CI for revenue.
Through these experiments, we confirm that the finite sample of $\beta$ converges to a normal distribution, 
with fast convergence in the entries whose constraints are tight. We confirm that for Hessian estimation, a suitable choice of $\varepsilon_t$
of smoothing parameter sequence is $t^{-d}$ for $d \in (.3, .5)$. Finally, the 
revenue confidence intervals in \cref{thm:variance_estimation} and \cref{eq:ab_rev_ci} attain the nominal coverage rate.

\section*{Acknowledgements}
This research was supported by the Office of Naval Research under grants N00014-22-1-2530 and N00014-23-1-2374, and the National Science Foundation award IIS-2147361.
\bibliography{refs.bib}
\bibliographystyle{icml2023}

\newpage
\appendix
\onecolumn
\section{Notations}
\begin{tabular}{r|l}
    \textbf{Symbol} & \textbf{Meaning} \\
    \hline
    \textbf{Notations in First-Price Pacing Equilibrium} & \\
    $\betast, \betagam$ & pacing multiplier \\
    $\pst(\cdot), \pgam$, $\REVst, \REVgam$& price and revenue \\
    $\must(\cdot), \mugam, \mubarst, \mubargam$ & utility generated from items \\
    $\deltast, \deltagam$ &  leftover budget \\
    $\ust, \ugam$ & total utility = utility from items + leftover budgets \\
    $\xst(\cdot), \xgam$ & allocation \\
    $s(\cdot)$ & supply (a probability density) \\
    $b$ & budget \\
    $v, v_i$ &  valuations\\
    \hline
    \textbf{Notations in Eisenberg-Gale Grogram} & \\
    $F(\theta,\beta)$, $\Psi(\beta)$, and $f(\theta, \beta)$ & see \cref{eq:pop_deg} \\
    $H, H_t$ &  objective function of Eisenberg-Gale convex programs \\
    $B$ & $B = (0,1]^n$ the domain of $H$ and $H_t$ \\
    $\cH$ & the Hessian matrix of $H$ at $\betast$ \\
    $\cP$ & matrix whose diagonal = $\indi(\betasti < 1)$ \\
    $\cH_B$ & $= \cP\cH\cP$ \\
    $I$, $\Ic$ & partition of buyers into those with $\betasti = 1$ $(I)$ and those not $(\Ic)$\\

    \end{tabular}
\section{Related Works} \label{sec:related_work}

\textbf{A/B testing in two-sided markets.}
Empirical studies by 
\citet{blake2014marketplace,fradkin2019simulation} demonstrate bias in experiments due to marketplace interference.
\citet{basse2016randomization} study the bias and variance of treatment effects under two randomization schemes for auction experiments.
\citet{bojinov2019time} study the estimation of causal quantities in time series experiments.
Some recent state-of-the-art designs are the multiple randomization designs \citep{liu2021trustworthy,johari2022experimental,bajari2021multiple} and the switch-back designs 
\citep{sneider2018experiment,hu2022switchback,li2022interference,bojinov2022design,NEURIPS2020_abd98725}.
The surveys by 
\citet{kohavi2017surprising,bojinov2022online} contain detailed accounts of A/B testing in internet markets.
See \citet{larsen2022statistical} for an extensive survey on statistical challenges in A/B testing.
Our paper focuses on A/B testing in first-price auction markets with the consideration of equilibrium effects, to the best of our knowledge this is the first work to consider market equilibrium effects in A/B testing.

\textbf{Pacing equilibrium.}
Pacing and throttling are two prevalent budget-management methods on ad auction platforms.
Here we focus on pacing methods since that is our setting.
In the first-price setting,
\citet{borgs2007dynamics} study first price auctions with budget constraints in a perturbed model, whose limit prices converge to those of an FPPE.
Building on the work of \citet{borgs2007dynamics}, \citet{conitzer2022pacing} introduce the FPPE model and discover several properties of FPPE such as shill-proofness and monotonicity in buyers, budgets and goods.
There it is also established that FPPE is closely related to the quasilinear Fisher market equilibrium~\citep{chen2007note,cole2017convex}.
\citet{gao2022infinite} propose an infinite-dimensional variant of the quasilinear Fisher market, which lays the probability foundation of the current paper.
\citet{gao2021online,liao2022dualaveraging} study online computation of the infinite-dimensional Fisher market equilibrium.
In the second-price setting,
\citet{balseiro2015repeated} investigate budget-management in second-price auctions through a fluid mean-field approximation;
\citet{balseiro2019learning} study adaptive pacing strategy from buyers' perspective in a stochastic continuous setting;
\citet{balseiro2021budget} study several budget smoothing methods including multiplicative pacing in a stochastic context;
\citet{conitzer2022multiplicative} study second price pacing equilibrium, and shows that the equilibria exist under fractional allocations.

\textbf{$M$-estimation when the parameter is on the boundary}
There is a long literature on the statistical properties of 
$M$-estimators when the parameter is on the boundary~\citep{geyer1994asymptotics,shapiro1990differential,shapiro1988sensitivity,shapiro1989asymptotic,shapiro1991asymptotic,shapiro1993asymptotic,shapiro2000asymptotics,andrews1999estimation,andrews2001testing,knight1999epi,knight2001limiting,knight2006asymptotic,knight2010asymptotic,dupacova1988asymptotic,dupavcova1991non,self1987asymptotic}.
Some recent works on the statistical inference theory for constrained $M$-estimation include \citet{li2022proximal,hong2020numerical,hsieh2022inference}.
Our work leverages \citet{shapiro1989asymptotic}, which develops a general set of conditions
for CLTs of constrained $M$-estimators when the sample function is nonsmooth.
Working under the specific model of FPPE, we build on and go beyond these 
contributions by deriving sufficient condition for asymptotic normality in FPPE, establishing local asymptotic minimax theory and developing valid inferential procedures.

\textbf{Statistical learning and inference with equilibrium effects}
\citet{Wager2021,munro2021treatment,sahoo2022policy}
take a mean-field game modeling approach and perform policy learning with a gradient descent method.
\citet{liao2022stat} consider statistical inference in the Fisher market equilibrium which is useful for fair and efficient resource allocations.
Statistical learning and inference 
has been investigated for other equilibrium models, such as
general exchange economy~\citep{guo2021online,liu2022welfare} and
matching markets~\citep{cen2022regret,dai2021learning,liu2021bandit,jagadeesan2021learning,min2022learn}.
Our work is also related to the rich literature of inference under interference~\citep{hudgens2008toward,aronow2017estimating,athey2018exact,leung2020treatment,hu2022average,li2022random}.
In the FPPE model, the interference among buyers is caused by the supply and budget constraint and the 
revenue-maximizing incentive of the platform.
In the economic literature, researchers have studied how to estimate auction market primitives from bid data; see \cite{athey2007nonparametric} for a survey.

\section{Omitted properties of FPPE}
\label{sec:fppe_properties}
\begin{defn}[Observed FPPE, formal]
    Given $(b,v,\sigma,\gamma)$, an observed FPPE is a tuple 
    $(\beta,x_1,\dots, x_n) \in [0,1]^n \times ([0, 1]^{t})^{n}$ 
 such that 
 \begin{itemize}
    \item (First-price) For all $\thetau$, $p_j (\thetau) = \max_i \betai v_i(\thetau)$. Moreover, $x_i(\thetau) > 0$ implies $\betai \vithe =\max_i \betai x_i (\thetau)$ for all $i$ and $\thetau$.
    \item (Supply and budget feasible)  For all $i$, $ \sigma \sumtau x_i(\thetau) p(\thetau)  \leq b_i$. For all $\thetau$, $\sumiton x_i (\thetau) \leq 1$.  
    \item (Revenue maximizing and market clearing) For all $i$, $ \sigma \sumtau x_i(\thetau) p(\thetau)  < b_i$ implies $\betai = 1$. For all $\thetau$, $p(\thetau) > 0$ implies $ \sumiton x_i(\thetau) = 1 $.
\end{itemize}

\end{defn}

We start by listing a number of known properties of the limit FPPE that we will use in our proofs.

A limit FPPE allocation $x$ and the limit FPPE pacing multiplier $\beta$ can be recovered through the convex programs, the population primal EG
\begin{align*}
    \label{eq:pop_eg}
    \max_{x \in L^\infty_+(\Theta), u, \delta \geq 0} \,
    \sumiton( b_i \log  (u_i) - \delta_i)
    \tag{\small{P-EG}}
    \\ \text{s.t. }     u_i   \leq  \langle v_i, s {x}_{i} \rangle + \delta_i 
    , \,
    \sumiton {x}_{i}(\theta) \leq 1
    \notag
\end{align*}
and the population dual EG
\begin{align*}
    \min_{0 < \betai \leq 1, i\in[n]}
     H(\beta)= \E[F(\theta, \beta)] = \E[ f(\beta,\theta) ] + \Psi(\theta)  
    \tag{\small P-DualEG}
\end{align*}
where $F (\theta, \beta) = f(\beta,\theta) + \Psi(\theta)$, $f(\theta,\beta) = \max_{i\in[n]} \beta_i v_i(\theta) $ and $\Psi(\beta) = -  \sumiton b_i \log \beta_i $.

\begin{lemma}[FPPE $\Leftrightarrow$ EG]
   The limit FPPE $\betast$ is the unique solution to \cref{eq:pop_deg}, and any limit FPPE $(\xst, \ust, \deltast)$ belongs to the set of optimal solutions to \cref{eq:pop_eg} \citep{gao2022infinite}.
\end{lemma}
\begin{lemma}[First-order conditions of limit FPPE, Theorem 10 from \citet{gao2022infinite}]
    \label{lm:first_order_limit_FPPE}
    Given $(b,v,s)$, the limit FPPE satisfies the following. 
\begin{itemize}
    \item $\usti = b_i / \betasti$ all $i$.
    \item $\betasti \vithe \leq \pst(\theta)$ all $i,\theta$.
    \item $\xsti(\theta), \deltasti , \betasti, \pst(\theta) \geq 0$ all $i,\theta$.
    \item $\pst(\theta) > 0 \implies \sumiton \xsti(\theta) = 1$ all $\theta$.
    \item $\deltasti > 0 \implies \betasti = 1$ all $i$.
    \item $\xsti(\theta)> 0 \implies \betasti = \pst(\theta) / \vithe$ all $i,\theta$.
\end{itemize}
\end{lemma}

\begin{lemma}[From EG to FPPE]
    \label{lm:from_eg_to_fppe}
    Recall the definition of $H$ in \cref{eq:pop_deg}. Under \nameref{as:smoothness}  
\begin{align*}
& \deltast = - \nabla H(\betast), 
\\
& \pst(\theta) = f(\theta,\betast), 
\\
& \must (\theta) = \xst(\theta)\odot v(\theta) = \nabla_\beta f (\theta, \betast)  ,
\\
& \mubarst = \E[\nabla f(\theta,\betast)] = \nabla \fbar (\betast) ,
\\
& \REVst= \fbar(\betast), \text{ and}
\\
& \NSWst = \Psi(\betast) + \sumiton b_i\log b_i.
\end{align*}
\end{lemma}

\begin{lemma}[Scale-invariance]\label{lm:scale_invariance_fppe}
Scaling the budget and values at the same time does not change the market equilibrium. 
Scaling the value and the supply inversely does not change the market equilibrium.
That is, given a positive scalar $\alpha$, if $(x,\beta,p) \in \FPPE(b,v,s)$, then 
\begin{align*}
    &   (x,\beta,  \alpha p) \in \FPPE(\alpha b, \alpha v,s)\;,
    \\
    &   (x, \beta, p ) \in \FPPE(b, \alpha v, \tfrac1\alpha s) \;.
\end{align*}

Similarly, for a given observed FPPE $\oFPPE(b,v,\sigma, \gamma)$ defined in \cref{def:observed_fppe}, and any positive scalar $\alpha$, if $    (x,\beta, p) \in \oFPPE (b, v, \sigma, \gamma)$, then 
\begin{align}
    (x,\beta, \alpha p) \in \oFPPE (\alpha b, \alpha v, \sigma, \gamma),
    \\
    (x,\beta , p) \in \oFPPE(  b, \alpha v, \tfrac1\alpha \sigma, \gamma)
\end{align}
\end{lemma}

\begin{lemma} \label{lm:p_equal_muHinvmu}
    Let $(\xst, \betast, \pst) = \FPPE(b,v,s)$. Assume \nameref{as:smoothness} holds. 
    Then $ \betasti < 1$ for all $i$ implies $\pst (\theta) = (\mubarst) \tp \cH \inv \must(\theta)$.
\end{lemma}
\begin{proof}
    If $\betasti < 1$ for all $i$, then $\mubarst = \ust$.
    By definition and \nameref{as:smoothness}, $\musti(\theta) = \xsti(\theta) \vithe$ and $\pst (\theta)= \max_i \betasti \vithe = \sumiton \xsti(\theta) \betasti \vithe $.
    It suffices to show 
    $(\mubarst )\tp \cH \inv = (\betast) \tp$, or equivalently $\cH \betast = \ust$.
    Recall $\fbar(\beta) = \E[\max_i\betai \vithe]$ is homogenous, i.e., $\fbar(\alpha\beta) = \alpha \fbar (\beta)$ for any positive scalar $\alpha \geq 0$.
    By the Euler's identity for homogenous functions, we have 
    \begin{align*}
        \nabla \fbar (\beta) = \sumiton \betai \nabla_i \fbar(\beta)
    \end{align*}
    Taking derivative again, we have $\nabla\sq\fbar(\beta) \beta = 0$ for all $\beta$.
    Finally, note $\cH = \nabla\sq\fbar(\betast) + \nabla \sq \Psi(\betast)$, we have $
    \cH \betast = \nabla\sq\Psi(\betast) \betast = \ust = \mubarst $,
    where the second equality holds by the first-order condition that $\usti = b_i / \betasti$, and 
    the third by the fact that $\betasti < 1$ for all $i$ implies $\deltasti = 0$. 
    So $(\mubarst) \tp\cH\inv \must(\theta) = (\betast) \tp \must(\theta) = \sumiton \betasti \musti(\theta) = \pst(\theta)$.
\end{proof}

As for the limit FPPE, the observed FPPE $\oFPPE(b, v, \sigma, \gamma)$ is characterized by primal and dual convex programs:
\begin{align}   
    \max_{x,\delta, u\geq 0}
    \bigg\{ 
       \sumiton(b_i \log(u_i) - \delta_i)
    \;\big|\;  
    u_i   \leq \sigma \langle v_i(\gamma), {x}_{i} \rangle + \delta_i \;
    \forall i 
    , \quad 
    \sumiton {x}_{i}^\tau \leq  1 \;  \forall \tau
    \bigg\} \;,
\\
    \min_{1_t \geq \beta > 0}
    \bigg\{ H_t(\beta) \defeq \sigma \sumtau \max_{i\in[n]} \beta_i v_i(\theta^\tau)  -   \sumiton b_i \log \beta_i 
    \bigg\} \;.
\end{align}
See \citet[Sec.\ 5]{conitzer2022pacing} and \citet{gao2022infinite} for more details on the convex program characterization of observed FPPE.
As with the limit FPPE, the observed FPPE has an analogous set of properties.
\begin{lemma}[First-order conditions of observed FPPE, from EG to FPPE]
    Given $ (b,v,\gamma)$, the observed FPPE $\oFPPE(b, v, \tfrac1t, \gamma)$ satisfies the following.
    \begin{itemize}
        \item $\ugami = b_i / \betagami$ all $i$.
        \item $\betagami \vithe \leq \pgam(\thetau)$ all $i,\thetau$.
        \item $\xgami(\thetau), \deltagami , \betagami, \pgam(\thetau) \geq 0$ all $i,\thetau$.
        \item $\pgam(\thetau) > 0 \implies \sumiton \xgami(\thetau) = 1 $ all $\thetau$.
        \item $\deltagami > 0 \implies \betagami = 1$ all $i$.
        \item $\xgami(\thetau)> 0 \implies \betagami = \pgam(\thetau) / v_i(\thetau)$ all $i,\thetau$.
    \end{itemize}
    
    Moreover, recall the definition of $H_t$ in \cref{eq:sample_deg}. Then  
    \begin{align*}
    &  \delta^\gam \in - \partial H_t(\betagam) , \mubargam \in \partial \bigg(\frac1t \sumtau f(\thetau,\betagam)\bigg), \pgam(\thetau) = f(\thetau,\betagam) , 
    \\
    & \mugam
    (\thetau) =  \xgam(\thetau)\odot v(\thetau) \in \partial_\beta f (\thetau, \betagam)  
    \\
    & \REVgam= \frac1t \sumtau f(\thetau,\betagam), \NSWgam = \Psi(\betagam) + \sumiton b_i\log b_i.
    \end{align*}
\end{lemma}

\subsection{Proof of \cref{lm:hessian_implications}}
\begin{proof}[Proof of \cref{lm:hessian_implications}]
    The proof follows a similar argument as in \citet{liao2022stat} for non-quasilinear Fisher market (recall that an FPPE is a quasilinear Fisher market). 
    Recall $f(\theta,\beta) = \max_i \beta_i \vithe$. 
    We define 
    \begin{align*}
        \epsilon(\theta,\beta) = \max_i \{\betai\vithe\} - \operatornamewithlimits{secondmax}_i\{\betai\vithe\}
    \end{align*}
    where $\operatorname{secondmax}$ is the second-highest entry (which could possibly be equal to the highest entry). 
    For example, $\operatorname{secondmax}(1,1,2) = 1$.

    Note $f(\theta,\cdot)$ is differentiable at $\beta$ if and only if $ \epsilon(\theta,\beta) > 0$, since this holds if and only if the subdifferential is a singleton at $\beta$. 
    Let $\Theta_{\text{diff}}(\beta) \defeq \{ \theta: f(\theta,\beta) \text{ differentiable at $\beta$} \}$
    and $\Thetatie(\beta) = \{\theta : \betai\vithe = \beta_k v_k(\theta) \text{ for some $k\neq i$} \}$. Then 
    \begin{align*}
        \Theta_{\text{diff}}(\beta)  = \{\theta: {\epsilon(\theta,\beta)} > 0 \} 
        \;, \Thetatie(\beta) = \{\theta: {\epsilon(\theta,\beta)} = 0 \} . 
    \end{align*}
    By Proposition 2.3 from~\citet{bertsekas1973stochastic} we know $\fbar(\beta) = \E[f(\theta,\beta)]$ is differentiable at $\beta$ if and only if $ \Theta_{\text{diff}}(\beta) $ is measure one.

    \emph{Proof of $\Thetatie$ is measure zero.} 
    Recall $\Thetatie = \Thetatie(\betast)$. The claim follows from the fact that the SMO assumption implies differentiability, which implies that the complement of $\Thetatie$ is measure one.

    \emph{Proof of uniqueness of $\xst, \must, \deltast$.}
    By \cref{lm:first_order_limit_FPPE}, we know $\xsti(\theta) > 0$ only if $\betasti \vithe \geq \betast_k v_k (\theta)$ for all other $k$. 
    Under the pacing profile $\betast$, for all (but a measure-zero set of) items there is only one winning bidder, i.e., for almost all $\theta$ there is a unique $i$ such that $\xsti (\theta) > 0$ and $\xst_k(\theta) = 0$ for all $ k \neq i$.
    Coupled with the limit FPPE first-order condition that $\sumiton \xsti (\theta) = 1$ we know $\xst$ is unique.
    This immediately implies uniqueness of $\must,\deltast$ as well.

    \emph{Proof of existence of $\Ndiff$.} By the assumption that $\fbar$ is twice continuously differentiable at $\betast$, there is a neighborhood $\Ndiff$ such that $\fbar$ is 
    continuously differentiable on $\Ndiff$. By the same argument as for $\Thetatie(\betast)$ being measure zero, $\Thetatie(\beta)$ is measure zero for each $\beta \in \Ndiff$. 

\end{proof}

\section{More A/B testing estimands} \label{sec:abtesting_more}
We remark that the potential value functions are suitable for modeling either item-side or 
buyer-side treatments.
In the context of ad auctions, item-size treatment are, for example, positions of the ads in the browser,
whether links are attached to the ads and so on.
Buyer-side treatments are, for example, a new layout of the ad campaign setup portal for the advertisers.
The following discussion centers around item-side treatments since they are more prominent 
in practice, but readers should keep in mind that our theory extends to 
buyer-side treatments.

As discussed in \cref{sec:convex_program}, each FPPE has a convex program characterization.
If the market is given treatment $w\in \{0,1\}$, then the limit FPPE pacing multipliers can be recovered by 
$
    \max_{\beta \in B} \int \max_i \betai v_i({w},\theta) s(\theta)\diff \theta  -\sumiton b_i \log \betai \;.
$
Let $\betast (w)$ be the unique solution to the above program and also the unique FPPE pacing multiplier. 
The FPPE prices and revenue are $\pst(w,\theta) \defeq\max_i v_i(w,\theta) \betasti(w)$ and $\REVst(w) \defeq \int \pst(w, \theta) s(\theta)\diff \theta$. 
The utility vector under treatment $w$ is 
$ \usti (w) \defeq b_i / \betast(w)$.
The NSW is $\NSWst(w) \defeq\sumiton b_i \log \usti(w)$.

Other metrics of treatment effects could be 
(i) treatment effects on revenue: $\tau_{\REV} \defeq \REVst(1) - \REVst(0) $,
(ii) treatment effects on Nash social welfare: $\tau_{\NSW}  \defeq \sumiton b_i \log \usti(1) -\sumiton b_i \log \usti(0)$,
(iii) treatment effects on pacing multiplier: $\tau_{\beta} \defeq \betast (1) - \betast (0)$, and 
(vi) treatment effects on utilities: $\tau_{u} \defeq \ust (1) - \ust(0)$.
The estimators of treatment effects are 
$
     \hat \tau_{\REV} \defeq \REVgam(1) - \REVgam(0) \;, \; $
   $ \hat \tau_{\beta} \defeq \betagam(1) - \betagam(0)\;, $
   $  \hat \tau_{ \NSW}  \defeq \NSWgam(1) - \NSWgam(0) \;, \; $ and 
    $ \hat \tau_{u} \defeq \ugam(1) - \ugam (0) \; .$

For given $(b, s)$, the (co)variances $\Sigma_\beta,\Sigma_u,\sigma\sq_\NSW, \sigma\sq_\REV$ in \cref{thm:clt} and \cref{cor:clt_u_and_nsw} are functionals of the value functions. We will use $\Sigma_\beta (w),\Sigma_u (w),\sigma\sq_\NSW (w), \sigma\sq_\REV(w)$ to represents the (co)variances in the market formed with value functions $\{v_i(w,\cdot)\}_i$.
Each variance can be estimated the same way as in \cref{sec:inference}.
Let $\betagam(w), \ugam(w)$, $\REVgam(w)$ and $\NSWgam(w)$ denote the observed FPPE quantities for treatment $w \in \{0,1\}$. 
\begin{theorem}[Treatment effects CLT]
Suppose \nameref{as:smoothness} and \nameref{as:constraint_qualification} hold in the limit markets $\FPPE(b,v(1),s)$ and $\FPPE(b,v(0),s)$. Then 
 $\sqrt{t} (\hat \tau_\REV - \tau_\REV) \tod\cN\allowbreak \big(0, \frac{ \sigma\sq_\REV(1)}{\pi} \allowbreak + \allowbreak\frac{ \sigma\sq_\REV(0)}{(1-\pi)} \big) $,
 $\sqrt{t} (\hat \tau_{\NSW} - \tau_{\NSW}) \tod \cN\allowbreak \big(0, \frac{ \sigma\sq_\NSW(1)}{\pi} \allowbreak + \allowbreak\frac{ \sigma\sq_\NSW(0)}{(1-\pi)}\big)$,
 $\sqrt{t} (\hat \tau_\beta - \tau_\beta) \tod \cN \big(0, \frac{1}{\pi} \Sigma_{\beta}(1) + \frac{1}{(1-\pi)} \Sigma_{\beta}(0)\big)$, and
 $\sqrt{t} (\hat \tau_u - \tau_u) \tod \cN\big(0, \frac{1}{\pi} \Sigma_{u}(1) + \frac{1}{(1-\pi)} \Sigma_{u}(0)\big)$.
Proof in \cref{sec:proof:thm:clt_ab_testing}.
\end{theorem}

\begin{algorithm}
    \caption{A/B test effect of a new feature on revenue} \label{alg}
\textbf{Step 1. Experiment.} 
Choose the new feature assignment probability $\pi$.
Perform A/B testing with budget splitting and item randomization.
Form two first-price pacing equilibria. 

\textbf{Step 2. Collect data.} Observe the equilibrium data from the two markets, including prices, item allocations, pacing multipliers, leftover budgets, and values of the observed items. 

\textbf{Step 3. Compute CI.} Compute the revenue variance as \cref{eq:rev_variance} for each market, obtaining $ \hat \sigma\sq_\REV(1)$ and $ \hat \sigma\sq_\REV(0)$, and form the confidence interval
\begin{align}
    \hat \tau_\REV \pm z_{\alpha/2} \bigg( \frac{ \hat \sigma\sq_\REV(1)}{\pi} \allowbreak + \allowbreak\frac{ \hat \sigma\sq_\REV(0)}{(1-\pi)}  \bigg)
    \label{eq:ab_rev_ci}
\end{align}
\end{algorithm}

\section{Technical Lemmas}

\subsection{A CLT for constrained $M$-estimator}
We introduce a CLT result from \cite{shapiro1989asymptotic} that handles $M$-estimation when the true parameter is on the boundary of the constraint set. 
Throughout this section, when we refer to assumptions A1, A2, B2, etc, we mean those assumptions in \citet{shapiro1989asymptotic}.

Let $(\Theta, P)$ be a probability space. Consider $f: \Theta \times \R^n \to \R$ and a set $B \subset \R^n$. Let $\theta_1, \ldots, \theta_t$ be a sample of independent random variables with values in $\Theta$ having the common probability distribution $P$.
Let $\phi( \beta) = P f(\cdot, \beta) = \E[f(\theta, \beta)]$, and $\psi_t(\beta) = P_t f(\cdot,\beta) = \frac1t \sum_{i=1}^t f(\theta_i, \beta)$. Let $\beta_0$ be the unique minimizer of $\phi$ over $B$  (Assumption A4 in  \citet{shapiro1989asymptotic}).
Let $\vartheta_ t = \inf _B \psi_t$ and $\hat \beta $ be an optimal solution.

We begin with some blanket assumptions.
Suppose the geometry of $B$ at $\beta_0$ is given by functions $g_i(\beta)$ (Assumption B1), i.e., there exists a neighborhood $N$ such that
$$
B \cap {N}=\left\{\beta \in {N}: g_i(\beta)=0, i \in K ; g_i(\beta) \leq 0, i \in J\right\},
$$
where $K$ and $J$ are finite index sets and the constraints in $J$ are active at $\beta_0$, meaning   $g_i\left(\beta_0\right)=0$ for all $i \in J$. Assume the functions $g_i, i \in K \cup J$, are twice continuously differentiable in a neighborhood of $\beta_0$ (Assumption B2).
Define the Lagrangian function by
$
l(\beta, \lambda)=\phi(\beta)+\sum_{i \in K \cup J} \lambda_i g_i(\beta).
$
Let $\Lambda_0$ be the set of optimal Lagrange multipliers, i.e., $\lambda \in \Lambda_0$ iff
$
\nabla l\left(\beta_0, \lambda\right)=0
$ (assuming differentiability)
and $\lambda_i \geq 0, i \in J$.

\begin{lemma}[Theorems 3.1 and 3.2 from \citet{shapiro1989asymptotic}] \label{thm:the_shapiro_thm}
    Assume there exists a neighborhood $N$ of $\beta_0$ such that the following holds.
    \begin{enumconditions}
        \item Conditions on the sample function $f$ and the distribution $P$.
        \begin{itemize}
            \item 
        (Assumption~A1 in the original paper) For almost every $\theta$, $f(\theta,\beta)$ is a continuous function of $\beta$, and for all $\beta\in B$, $f(\theta,\beta)$ is a measurable function of $\theta$.
        \item 
        (Assumption~A2) The family $\{f(\theta, \beta)\}, \beta \in B$, is uniformly integrable.
        \item 
        (Assumption~A4) For all $\theta$, there exist a positive constant $K(\theta)$ such that
        $
        |f(\theta, w)-f(\theta, \beta)| \leq K(\theta)\|w-\beta\|
        $
        for all $\beta, w \in {N}$.
        \item 
        (Assumption~A5) For each fixed $\beta \in {N}, f(\theta, \cdot)$ is continuously differentiable at $\beta$ for almost every $\theta$.
        \item
        (Assumption~A6) The family $\{\nabla f(\theta, \beta)\}_{\beta \in {N}}$, is uniformly integrable.
        \item
        (Assumption~D) The expectation $\E[\left\|\nabla f\left(\theta, \beta_0\right)\right\|^2]$ is finite. 
        \item
        (Assumption~B4) The function $\phi$ is twice continuously differentiable in a neighborhood of $\beta_0$.
        \end{itemize}

        \item Conditions on the optimal solution.
        \begin{itemize}
            \item 
        (Assumption~B3) A constraint qualification, the Mangasarian-Fromovitz condition: The gradient vectors $\nabla g_i\left(\beta_0\right), i \in K$, are linearly independent,
        and there exists a vector $w$ such that
        $
        w \tp \nabla g_i\left(\beta_0\right)=0,  i \in K$ and $w \tp \nabla g_i\left(\beta_0\right)<0,  i \in J .
        $
        \item
        (Assumption~B5) Second-order sufficient conditions: 
        Let $C$ be the cone of critical directions
        \begin{align}
            \label{eq:def_C}
            C=\left\{w: w \tp \nabla g_i\left(\beta_0\right)=0, i \in K ; w \tp \nabla g_i\left(\beta_0\right) \leq 0, i \in J ; w \tp \nabla \phi\left(\beta_0\right) \leq 0\right\} .
        \end{align}
        The assumption requires that for all nonzero $w \in C$,
        $
        \max _{\lambda \in \Lambda_0} w \tp \nabla^2 l\left(\beta_0, \lambda\right) w>0,
        $
        \end{itemize}

        \item Stochastic equicontinuity, a modified version of Assumption~C1 in the original paper.
            For any sequence $\delta_t = o(1)$, the variable
            \begin{align}
                \label{lm:se_of_subgradient}
                \sup _{\beta:\| \beta -\beta_0\| \leq\delta_t} \frac{\left\|(\nabla \psi_t - \nabla\phi)(\beta)
                -(\nabla \psi_t - \nabla\phi)(\beta_0)
                \right\|}{t^{-1 / 2}+\left\|\beta-\beta_0\right\|} = o_p(1)
            \end{align}
            as $t \rightarrow \infty$. Here the supremum is taken over $\beta$ such that $\nabla \psi_t(\beta)$ exists. 
            \label{it:stoc_equic}
          
    \end{enumconditions}
    Then it holds that $ \hat\beta \toprob \beta_0$.
    Let
    \begin{align}
        \label{eq:def_zeta}
        \zeta_t=\nabla \psi_t\left(\beta_0\right)-\nabla \phi\left(\beta_0\right),
    \end{align}
    and 
    \begin{align}
        \label{eq:def_q}
        q(w)=\max_{  \lambda \in \Lambda_0} \{w \tp  \nabla^2 l\left(\beta_0, \lambda\right)w \}.
    \end{align}
    Then 
    $$\vartheta_t\defeq \inf _B \psi_t =\psi_t\left(\beta_0\right)+\min _{w \in C}\{w \tp \zeta_t +\frac{1}{2} q(w)\}+o_p (t^{-1} ).$$ 
    Furthermore, suppose for all $\zeta$ the function $w \mapsto w \tp \zeta+\frac{1}{2} q(w)$ has a unique minimizer $\bar{\omega}(\zeta)$ over $C$. Then $$\|\hat {\beta} -\beta_0-\bar{\omega}(\zeta_t)\|=o_p(t^{-1 / 2}) .$$
\end{lemma}

\begin{remark}[The stochastic equicontinuity condition]
    By inspecting the proof, the original Assumption~C1, 
    $\sup _{\beta \in B \cap {N}} {\left\|\nabla \psi_t(\beta)-\nabla \phi(\beta)-\nabla \psi_t\left(\beta_0\right)+\nabla \phi\left(\beta_0\right)\right\|}/[{t^{-1 / 2}+\left\|\beta-\beta_0\right\|}] = o_p(1)$, which requires uniform convergence over a fixed neighborhood $N$, 
    can be relaxed to the uniform convergence in a shrinking neighborhood of $\beta_0$.
    The shrinking neighborhood condition is in fact
    standard, see, e.g., \citet{pakes1989simulation,newey1994large}.
\end{remark}

\begin{remark}
    The limit distribution of the minimizer is characterized by three objects: the limit distribution of $\zeta_t$ defined in \cref{eq:def_zeta}, the critical cone $C$ defined in \cref{eq:def_C} and the piecewise quadratic function $q$ defined in \cref{eq:def_q}.
\end{remark}


Hessian matrix estimation at the optimum $\beta_0$ can be done via the numerical difference method.

\begin{lemma}[Hessian estimation via numerical difference, adapted from Theorem 7.4 from \citet{newey1994large}]
    \label{lm:hessian_estiamtion}
    Recall $\phi(\beta) = Pf(\cdot, \beta)$, $\psi_t(\beta) = P_t f(\cdot, \beta)$ and $\zeta_t=\nabla \psi_t(\beta_0) - \nabla \phi(\beta_0)$. We are interested in the Hessian matrix $H = \nabla\sq\phi(\beta_0)$.
    Let $\beta_0$ be any point and let $\hat \beta$ be an estimate of $\beta_0$. 
    Assume 
    \begin{enumconditions}
    \item $\hat \beta - \beta_0 = O_p(t^{-1/2})$; \label{it:1} 
    \item $\phi$ is twice differentiable at $\beta_0$ with non-singular Hessian matrix $H$; 
    \label{it:2} 
    \item $\sqrt{t} \zeta_t \tod N(0, \Omega)$ for some matrix $\Omega$; 
    \label{it:3} 
    \item for any positive sequence $\delta_t = o(1)$,
    the stochastic equicontinuity condition \cref{lm:se_of_subgradient} holds. 
    \label{it:4} 
    \end{enumconditions}
    Suppose $\varepsilon_t \to 0$ and $\varepsilon_t \sqrt{t} \to \infty$. Then $\hat H \toprob H$, where $\hat H$ is the numerical difference estimator whose $(i,j)$-th entry is
\begin{align*} 
    \hat{H}_{i j}=& {[\psi_t(\hat{\beta}+e_i \varepsilon_t+e_j \varepsilon_t)-\psi_t(\hat{\beta}-e_i \varepsilon_t+
    {e}_j \varepsilon_t)-\psi_t(\hat{\beta}+e_i \varepsilon_t-e_j \varepsilon_t).} \\ 
    &+\psi_t(\hat{\beta}-e_i \varepsilon_t-e_j \varepsilon_t)] / 4 \varepsilon_t^2. 
\end{align*}
\end{lemma}
\begin{proof}[Proof of \cref{lm:hessian_estiamtion}]
    We provide a proof sketch following Theorem 7.4 from \citet{newey1994large} and 
    Lemma 3.3 in \citet{shapiro1989asymptotic}.
    By \cref{it:1} and $\invtroot = o(\varepsilon_t)$ we know for any vector $a\in\Rn$, it holds $\| \hat \beta + \varepsilon_t a - \beta_0 \| = O_p(\varepsilon_t)$. 
    Let $\beta = \hat \beta + a \varepsilon_t$.
    By a mean value theorem for locally Lipschitz functions (see \citet{Clarke1990}; the lemma is also used in the proof of Lemma 3.3 in \citet{shapiro1989asymptotic}),
    there is a (sample-path dependent) $\beta'$ on the segment joining $\beta$ and $\beta_0$ such that 
    $$ (\psi_t - \phi)(\beta) - (\psi_t - \phi)(\beta_0) = (\zeta\st_t)\tp (\beta - \beta_0).$$
    for some $\zeta\st_t \in \partial \psi_t(\beta') - \nabla \phi(\beta')$.
    Then 
    \begin{align}
        |(\psi_t - \phi)(\beta) - (\psi_t - \phi)(\beta_0)| 
        \notag
        &\leq \|\zeta_t \| \|\beta-\beta_0\| + \| \zeta\st_t - \zeta_t\|\|\beta-\beta_0\|
        \\
        & = \|\zeta_t \| \|\beta-\beta_0\| + o_p(\invtroot + \|\beta' - \beta_0 \|)\|\beta-\beta_0\| 
        \tag{by \cref{it:4} }
        \\
        &= O_p(\invtroot) O_p(\varepsilon_t) + o_p(\invtroot + O_p(\varepsilon_t)) O_p(\varepsilon_t ) \tag{by \cref{it:3} }
        \\ &= o_p(\varepsilon_t\sq) 
        \label{eq:from_psi_to_phi}
    \end{align}
    Next by \cref{it:2} we have a quadratic expansion
    \begin{align}
        \label{eq:quad_expan}
        \phi(\beta) - \phi(\beta_0) - \nabla \phi(\beta_0)\tp (\beta-\beta_0) - \frac12  (\beta-\beta_0) \tp H  (\beta-\beta_0) = o_p(\varepsilon_t\sq). 
    \end{align}
    Let $a_{\pm\pm} = \pm e_i \varepsilon_t \pm e_j \varepsilon_t$, $\hat\beta_{\pm\pm} = \hat\beta + a_{\pm \pm}$ and $ d_{\pm\pm} = \hat\beta_{\pm\pm} -\beta_0$.
    Then $d_{\pm\pm} = O_p(\varepsilon_t)$ and $d_{\pm\pm} = a_{\pm\pm} + o_p(\varepsilon_t)$.
    Applying the above bounds with $\beta \leftarrow \hat\beta_{\pm \pm} $, recalling the definition of $\hat H_{ij}$, we have 
    \begin{align}
        \hat H_{ij} &= [\psi_t(\hat\beta_{++})
        -\psi_t(\hat\beta_{-+})
        -\psi_t(\hat\beta_{+-})
        +\psi_t(\hat\beta_{--})] / (4\varepsilon_t\sq)
        \notag
        \\ 
        &=[\phi(\hat\beta_{++})
        -\phi(\hat\beta_{-+})
        -\phi(\hat\beta_{+-})
        +\phi(\hat\beta_{--}) + o_p(\varepsilon_t\sq)] / (4\varepsilon_t\sq)
       \tag{by \cref{eq:from_psi_to_phi}}
        \\ 
        &= [\nabla\phi(\beta_0) \tp (d_{++} - d_{-+} - d_{+-} + d_{--}) 
        \notag
        \\  & \quad
        + 
        \frac12( d_{++} \tp H d_{++} 
        - d_{+-} \tp H d_{+-}
        - d_{-+} \tp H d_{-+}
        + d_{--} \tp H d_{--}  ) + o_p(\varepsilon_t\sq) ] / 
        (4\varepsilon_t\sq)
        \tag{by \cref{eq:quad_expan}}
        \\
        & = [0 + 
        \frac12(
          a_{++} \tp H a_{++} 
        - a_{-+} \tp H a_{-+}
        - a_{+-} \tp H a_{+-}
        + a_{--} \tp H a_{--}  ) 
        + o_p(\varepsilon_t\sq) ] / 
        (4\varepsilon_t\sq)
        \notag
        \\
        &  = [
        4 \varepsilon_t\sq H_{ij}   + o_p(\varepsilon_t\sq) ] / 
        (4\varepsilon_t\sq)
        \notag
        \\ 
        &= H_{ij} + \frac{o_p(\varepsilon_t\sq)}{\varepsilon_t\sq} = H_{ij} +o_p(1).
        \notag
    \end{align}
    In the above we use $d_{++}\tp H d_{++} = (d_{++}-a_{++})\tp H d_{++}  + (d_{++}-a_{++})\tp H a_{++}+ a_{++}\tp H a_{++} = o_p(\varepsilon_t\sq) + a_{++}\tp H a_{++}$, and similarly for other terms.
    This completes the proof of \cref{lm:hessian_estiamtion}.
\end{proof}

The original conditions for \cref{lm:hessian_estiamtion} in \citet{newey1994large}
require the true parameter $\beta_0$ to lie in the interior of $B$. However, this condition is only used to derive the bound $\hat \beta - \beta_0 = O_p(t^{-1/2})$, which is assumed in our adapted version.

\subsection{Showing stochastic equicontinuity via VC-subgraph function classes}

Next we review classical results from the empirical process literature \citep{vaart1996weak,gine2021mathematical}.

We begin with the notions of Donsker function class and stochastic equicontinuity.

Let $(\Theta,P)$ be a probability space. 
Let $\cF$ be a class of measurable functions of finite second moment.
The class $\cF$ is called $P$-Donsker if 
a certain central limit theorem holds for the class of random variables $\{ \sqrt{t} (P_t - P) f : f\in\cF\}$,
where 
$P_t f = \frac1t \sum_{i=1}^t f(X_i)$ where $X_i$'s are i.i.d.\ draws from $P$.
Because Donskerness will be used as an intermediate step that we will not actually need to show directly or utilize directly,
we refer the reader to Definition~3.7.29 from \citet{gine2021mathematical} for a precise definition.


\begin{lemma}[Donskerness $\Leftrightarrow$ stochastic equicontinuity, Theorem~3.7.31 from \citet{gine2021mathematical}]\label{thm:donsker_se}
    \label{lm:donsker_to_SE}
 Let $d_P^2(f, g)=P(f-g)^2-(P(f-g))^2$ and consider the pseudo-metric space $\left(\mathcal{F}, d_P\right)$.
 Assume $\cF$ satisfies the condition $\sup_{f\in\cF}|f(x) - Pf | <\infty $ for all $x\in \Theta$. 
 Then the following are equivalent 
 \begin{itemize}
    \item $\cF$ is $P$-Donsker.
    \item $\left(\mathcal{F}, d_P\right)$ is totally bounded, and stochastic equicontinuity under the L2 function norm holds, i.e., for any $\delta_t = o(1)$,  
    $$
    \sup_{(f,g) \in [\delta_t]_{L^2}}|\sqrt{t}\left(P_t-P\right)(f-g)| = o_p(1)
    $$
    as $t\to \infty$,
    where $[\delta_t]_{L^2} = \{ (f,g):f, g \in \mathcal{F}, d_P(f, g) \leq \delta_t\}$.
 \end{itemize}
\end{lemma}
\cref{lm:donsker_to_SE} reduces the problem of showing stochastic equicontinuity under the L2 function norm to showing Donskerness.
In order to show Donskerness, we will show that our function class is \emph{VC-subgraph}, which implies Donskerness.
At the end, we will connect stochastic equicontinuity under the L2 function norm to the stochastic equicontinuity that we need (see \cref{lm:ltwo_to_parameter}).

Let $\mathcal{C}$ be a class of subsets of a set $\Theta$. Let $A \subseteq \Theta$ be a finite set. We say that $\mathcal{C}$ \emph{shatters} $A$ if every subset of $A$ is the intersection of $A$ with some set $C \in \mathcal{C}$. The subgraph of a real function $f$ on $\Theta$ is the set
$
G_f=\{(s, t): s \in \Theta, t \in \mathbb{R}, t \leq f(s)\} .
$

\begin{defn}[VC-subgraph function classes, Definition~3.6.1 and 3.6.8 from \citet{gine2021mathematical}]
A collection of sets $\mathcal{C}$ is a Vapnik-\v Cervonenkis class $(\mathcal{C}$ is $V C)$ if there exists $k<\infty$ such that $\mathcal{C}$ does not shatter any subsets of $\Theta$ of cardinality $k$.
A class of functions $\mathcal{F}$ is $V C$-subgraph if the class of sets $\mathcal{C}=\left\{G_f: f \in \mathcal{F}\right\}$ is $V C$.
\end{defn}

\begin{lemma} [VC subgraph + envelop square integrability $\implies$ Donskerness, Theorem 3.7.37 from \citet{gine2021mathematical}]
    \label{lm:vc_to_donsker}

    If $\cF$ is VC-subgraph, and there exists a measurable $F$ such that $f\leq F$ for all $f\in \cF$ with 
    $P F\sq < \infty$, then $\cF$ is $P$-Donsker.
\end{lemma}

Since VC-subgraph implies Donskerness which is equivalent to stochastic equicontinuity, our problem reduces to showing the VC-subgraph property. The following lemmas show how to construct complex VC-subgraph function classes from simpler ones, and will be used in our proof.
\begin{lemma}[Preservation of VC class of sets, Lemma 2.6.17 from \citet{vaart1996weak}] \label{lm:CV_set_intersection}
    If $\cC$ and $\cD$ are VC classes of sets. Then $\cC \cap \cD = \{ C \cap D: C\in \cC, D \in \cD \}$ is VC. 
\end{lemma}

\begin{lemma}[Preservation of VC-subgraph function classes, Lemma~2.6.18 from \citet{vaart1996weak}]
    \label{lm:vc_preservation}
    Let $\mathcal{F}$ and $\mathcal{G}$ be $V C$-subgraph classes of functions on a set $\Theta$ and $g: \Theta \mapsto \mathbb{R}$ be a fixed function. Then
    $\mathcal{F} \vee \mathcal{G}= \{f \vee g: f \in \mathcal{F}, g \in \mathcal{G}\}$,
    $\mathcal{F}+g=\{f+g: f \in \mathcal{F}\}$,
    $\mathcal{F} \circ \phi = \{f \circ \phi : f\in \mathcal{F}\}$ is VC-subgraph for fixed $\phi: \cX \to \Theta$, 
    and $\mathcal{F} \cdot g=\{f g: f \in \mathcal{F}\}$ are VC-subgraph;
\end{lemma}
\begin{lemma}[Problem 9 Section 2.6 from \citet{vaart1996weak}]
    \label{lm:vcset_to_indicators}
    If a collection of sets $\mathcal{C}$ is a VC-class, then the collection of indicators of sets in $\mathcal{C}$ is a VC-subgraph class of the same index.
\end{lemma}
In general, the VC-subgraph property is not preserved by multiplication, whereas Donskerness is. Thus, our proof will use the VC-subgraph property up until a final step where we need to invoke multiplication, which will instead be applied on the Donskerness property.
\begin{lemma}[Corollary 9.32 from \citet{kosorok2008introduction}]
    \label{lm:donsker_multiplication}
    Let $\cF$ and $\cG$ be Donsker, 
    then $\cF \cdot \cG$ is Donsker if both $\cF$ and $\cG$ are uniformly bounded.
\end{lemma}

For parametric function classes, if the parametrization is continuous in
a certain sense, then
stochastic equicontinuity holds w.r.t.\ the norm in the parameter space.

\begin{lemma}[From $L^2$-norm to parameter norm, Lemma 2.17 from \citet{pakes1989simulation}; see also Lemma 1 from \citet{chen2003estimation}]
    \label{lm:ltwo_to_parameter}
    Suppose the function class $\cF = \{ f(\cdot, \beta), \beta\in B\}$, $B \subset \Rn$, is $P$-Donsker, with an envelope $F$ such that $P F\sq < \infty$. Suppose $\int\left[f(\cdot, \beta)-f\left(\cdot, \beta_0\right)\right]^2 d P \rightarrow 0 $ as $ \beta \rightarrow \beta_0$. Then for any positive sequence $\delta_t = o(1)$, it holds
    \begin{align}
        \label{eq:se_in_parameter}
        \sup _{\beta: \|\beta-\beta_0 \|<\delta_t}|
            \sqrt{t}\left(P_t-P\right)(            
            f(\cdot, \beta)- f\left(\cdot, \beta_0 \right))|=o_p(1).
    \end{align}
\end{lemma}

\begin{lemma}[\citet{andrews1994asymptotics}]
    \label{lm:se}
    If for any $\delta_t = o(1)$ \cref{eq:se_in_parameter} holds,
    then for any random elements $\beta_t$ such that $\| \beta_t - \beta_0\|_2 =o_p(1)$, it holds 
    $ 
        \sqrt{t}(P_t - P) ( f(\cdot,\beta_t) - f(\cdot, \beta_0)) = o_p(1).
    $
    
\end{lemma}

\begin{lemma}\label{lm:verifying_VC}
    Given any $n$ fixed functions $v_i: \Theta \to \R$, $i\in [n]$,
    the following two function classes
    \begin{align*}
    \cF_1 &= \{ \theta \mapsto  \max_i \{\beta_1v_1(\theta), \dots, \beta_n v_n(\theta)\}  :\beta \in B \}
    \\
    \cF_2 &= \{ \theta \mapsto  [D_{f,1}, \dots, D_{f,n}](\theta, \beta): \beta \in B \}  
    \end{align*}
    are VC-subgraph and Donsker.    
    Here $D_{f,i} (\theta, \beta) = v_i(\theta)\prod_{k=1}^n \indi( \betai \vithe \geq \beta_k v_k(\theta) )$ and $B=[0,1]^n$
\end{lemma}
\begin{proof}[Proof of \cref{lm:verifying_VC}]
    We show $\cF_1$ is VC-subgraph.
    For each $i$, the class 
    $\{ \theta \mapsto v_i(\theta) \beta_i : \beta_i \in [0,1] \}$ is VC-subgraph (Proposition 4.20 from \citet{wainwright2019high}, and Example 19.17 from \citet{van2000asymptotic}). 
    By the fact the VC-subgraph function classes are preserved by pairwise maximum (\cref{lm:vc_preservation}), we know $\cF_1$ is VC-subgraph. 
    Moreover, the required envelop condition holds since $\esssup _\theta f \leq \vbar $ for all $f\in \cF_1$, so $\cF_1$ is Donsker by \cref{lm:vc_to_donsker}.

    We now show $\cF_2$ is VC-subgraph.
    For a vector-valued function class, we say it is VC-subgraph if each coordinate is VC-subgraph.
    First, the class of sets $\{ \{ v \in \Rn:\betai v_i \geq \beta_k v_k\} \subset \Rn : \beta \in B \}$ is VC, for all $k \neq i$. 
    By \cref{lm:CV_set_intersection}, we know 
    the class of sets $\{ \{v \in \Rn : \betai v_i \geq \beta_k v_k, \forall k\neq i\} \subset \Rn: \beta \in B \}$ is VC.
    By \cref{lm:vcset_to_indicators}, we obtain that the class 
    $\{ \theta \mapsto \prod_{k=1}^n  \indi( \betai \vithe \geq \beta_k v_k(\theta) ): \beta \in B \}$
    is VC-subgraph.
    Finally, multiplying all functions by a fixed function preserves VC-subgraph classes (\cref{lm:vc_preservation}), and so 
    $\{ \theta \mapsto v_i(\theta)\prod_{k=1}^n \indi( \betai \vithe \geq \beta_k v_k(\theta) ): \beta \in B\}$ is VC-subgraph.
    Repeat the argument for each coordinate, and we obtain that $\cF_2$ is VC-subgraph.
    Moreover, the required envelop condition holds since $\esssup_\theta \|D_F(\theta, \beta)\|_2 \leq n\vbar $ for all $D_f \in \cF_2$, and so $\cF_2$ is Donsker by \cref{lm:vc_to_donsker}. We conclude the proof of \cref{lm:verifying_VC}.
\end{proof}

\section{Proofs for Main Theorems}

\subsection{Proof of \cref{thm:clt}} \label{sec:proof:thm:clt}
\begin{proof}[Proof of \cref{thm:clt}]
    We verify all the conditions in \cref{thm:the_shapiro_thm}. Recall $I = \{i:\betasti = 1\}$ is the set of active constraints. The local geometry of $B$ at $\betast$ is described by the $|I|$ constraint functions $g_i(\beta) = e_i\tp \beta - 1$, $i\in I$.

    First, we verify the conditions on the probability distribution and the sample function. 
    A1 holds obviously for the map $\beta \mapsto \max_i \betai v_i(\theta)$. 
    A2 holds by $f \leq \vbar$.
    A4 holds with Lipschitz constant $\vbar$.
    A5 holds since by \nameref{as:smoothness} there is a neighborhood $\Ndiff$ of $\betast$ such that for all $\beta \in \Ndiff$, the set $\{\theta: f(\theta,\cdot) \text{ not differentiable at $\beta$} \}$ is measure zero (cf. \cref{lm:hessian_implications}). 
    A6 holds by $\|\nabla f(\theta,\beta)\|_2\leq  n \vbar$.
    B4 holds by \nameref{as:smoothness}.

    Second, we verify the conditions on the optimality. 
    B3 holds since the constraint functions are $g_i(\beta) = e_i\tp \beta - 1$, $i\in I$, whose gradient vectors are obviously linear independent. Moreover, the set $\{\beta: \beta_i >0, i\in I\}$ is nonempty.
    B5 holds by $\nabla\sq H(\betast) = \nabla\sq\fbar (\betast) + \Diag(b_i/\betasti\sq) \succcurlyeq \Diag(b_i/\betasti\sq)$ being positive definite.

    Finally, we verify the stochastic equicontinuity condition. 
    Recall the definitions of the following two function classes from \cref{lm:verifying_VC} 
    \begin{align*}
        \cF_1 &= \{ \theta \mapsto  \max_i \{\beta_1v_1(\theta), \dots, \beta_n v_n(\theta)\}  :\beta \in B \},
        \\
        \cF_2 &= \{ \theta \mapsto  D_f (\theta,\beta): \beta \in B \}.
        \end{align*}    
        Here $B=[0,1]^n$.
    For any $\beta\in \Ndiff$ we have $\nabla f(\cdot , \beta) = D_f(\cdot,\beta) \in \cF_2$. In \cref{lm:verifying_VC} we show that $\cF_2$ is VC-subgraph and Donsker. By \cref{lm:vc_to_donsker} we know that a stochastic equicontinuity condition w.r.t.\ the $L^2$ norm holds, i.e., 
    \begin{align} 
        \label{eq:nabla_f_se}
        \sup_{\beta \in [\delta_t]_{L^2}} \nu_t ( D_f(\cdot, \beta) - D_f(\cdot, \betast) ) = o_p(1)
    \end{align}
    where $[\delta_t]_{L^2}  = \{ \beta: \beta \in \Ndiff, \int \|D_f(\cdot, \beta) - D_f(\cdot, \betast)\|_2 \sq \diff S \leq \delta_t\}$, $\int D_f \diff S =\int D_f s\diff \theta$,
    $\nu_t D_f = \invtroot (P_t - P) D_f= \invtroot \sumtau( D_f(\thetau) - \int D_f\diff S)$.
    Next, we note for (almost every) fixed $\theta$, $ \lim_{\beta \to \betast}\|D_f(\theta , \beta)  - D_f(\theta , \betast) \|\sq =0$ by $\Thetatie(\betast)$ is measure zero (a condition implied by \nameref{as:smoothness}). 
    Moreover, note 
    \begin{align*}
       \lim_{\beta \to \betast} \E[\|D_f(\theta , \beta)  -D_f(\theta , \betast) \|_2 \sq ] = 
       \E \big[ \lim_{\beta \to \betast}\|D_f(\theta , \beta)  - D_f(\theta , \betast) \|_2\sq \big] = 0
    \end{align*}
    where the exchange of limit and expectation is justified by bounded convergence theorem, and
    by \cref{lm:ltwo_to_parameter}, we can replace $[\delta_t]_{L^2}$ with $ [\delta_t] = \{\beta: \beta \in \Ndiff, \|\beta - \betast\|_2\leq \delta_t\}$ in \cref{eq:nabla_f_se}. Finally, note $\nabla  \fbar (\betast) = \E[D_f(\theta,\betast)]$, and 
    if $H_t$ is differentiable at $\beta \in \Ndiff$, then $\nabla f (\thetau,\beta) = D_f(\thetau, \beta)$ for all $\tau\in[t]$. Then 

    \begin{align}
        & \sup_{[\delta_t] \cap \{ \nabla H_t(\beta)\text{ exists} \}} \frac{\| (\nabla H_t - \nabla H) (\beta) - (\nabla H_t - \nabla H) (\betast) \| _ 2 }{\invtroot +  \|\beta - \betast\|_2 } 
        \notag
        \\ & = 
        \sup_{[\delta_t] \cap \{ \nabla H_t(\beta)\text{ exists}  \} } \frac{\| (P_t - P) D_f(\cdot,\beta) - (P_t - P) D_f(\cdot,\betast) \| _ 2 }{\invtroot +  \|\beta - \betast\|_2 } 
        \\
        & \leq 
        \sup_{[\delta_t]} \sqrt{t}  \| (P_t - P) D_f(\cdot,\beta) - (P_t - P) D_f(\cdot,\betast) \| _ 2 
        = o_p(1)  \tag{by \cref{eq:nabla_f_se}}
    \end{align}
    and thus the required stochastic equicontinuity condition holds. 

    Now we are ready to invoke \cref{thm:the_shapiro_thm}. We need to find the 
    three objects, $C, q, \zeta_t$ as in the lemma that characterize the limit distribution.
    The critical cone $C$ is 
    \begin{align}
        C &= \{w \in \Rn: w\tp e_i = 0 \text{ if $i\in I$ and $\deltasti > 0$},\quad 
        w \tp e_i \leq  0 \text{ if $i \in I$ and $\deltasti = 0$} \}
        \notag
        \\
       & = \{ w: Aw = 0\} \tag{\nameref{as:constraint_qualification}}
    \end{align}
    where $A \in \R^{|I| \times n}$ whose rows are $\{ e_i \tp, i \in I\}$. From here we can see the role of \nameref{as:constraint_qualification} is to ensure the 
    critical cone is a hyperplane, which ensures asymptotic normality of $\betagam$.

    If $|I| = 0$, i.e., $\betast$ lies in the interior of $B$, then $\cP$ is identity matrix, and the limit distribution is ture. 

    Now assume $|I| \geq 1$.
    Note $A A\tp $ is an identity matrix of size ${|I|}$ and $A\tp A 
    = \Diag(\indi(i \in I))  = \Diag(\indi (\betasti = 1))$.
        The optimal Lagrangian multiplier is unique and so the piecewise quadratic function $q$ is $q(w) = w\tp \cH w$. Finally, the gradient error term is 
    \begin{align}
        \zeta_t = \frac1t \sumtau\big( \xst(\thetau) \odot v(\thetau) - \mubarst \big).
    \end{align}
    The unique minimizer of $w\mapsto \frac12 w\tp \cH w + \zeta w$ over $\{w:Aw=0\}$ is $
    - (\cP \cH \cP)^\dagger  \zeta$ where 
    \begin{align*}
        \cP = I_n - A\tp (A A\tp ) ^{\dagger} A = \Diag( \indi(i \in \Icc))= \Diag(\indi (\betasti < 1)).
    \end{align*}
    For completeness, we provide details for solving this quadratic problem. 
    By writing down the KKT conditions, the optimality condition is 
    \begin{align*}
        \begin{bmatrix}
            \cH & A\tp 
            \\
            A & 0
        \end{bmatrix} 
        \begin{bmatrix}
            w 
            \\
            \lambda 
        \end{bmatrix}
        =
        \begin{bmatrix}
            -\zeta 
            \\
            0
        \end{bmatrix}
        \implies 
        \begin{bmatrix}
            w 
            \\
            \lambda 
        \end{bmatrix} =
        \begin{bmatrix}
           - (\cH\inv - \cH\inv A\tp (A\cH\inv A\tp)\inv A \cH \inv) \zeta 
            \\
           - ((A\cH\inv A\tp)\inv A \cH \inv) \zeta
        \end{bmatrix}
    \end{align*}
    where $\lambda \in \R^{|I|}$ is the Lagrangian multiplier.
    By a matrix equality, for any symmetric positive definite $\cH$ of size $n$ and $A\in \R^{|I|\times n}$ of rank $|I|$, it holds
    \begin{align}
        \cH\inv - \cH\inv A\tp (A\cH\inv A\tp)\inv A \cH \inv =  P_A ( P_A \cH  P_A)^\dagger  P_A = ( P_A\cH P_A)\pinv
    \end{align}
    with $ P_A = I_n - A\tp (A A\tp ) ^{\dagger} A$.
    We conclude that the asymptotic expansion 
    \begin{align}
        \label{eq:expansion_beta}
        \sqrt{t} (\betagam - \betast) =  \frac{1}{\sqrt t} \sumtau D_\beta(\thetau) + o_p(1)
    \end{align}
    holds, where $$D_\beta(\theta) = - (\cP \cH \cP)^\dagger  (\xst(\theta) \odot v(\theta) - \mubarst ) ,$$ and that the asymptotic distribution of $\sqrt{t}(\betagam - \betast)$ is $\cN(0, 
    \Sigma_\beta)$ with $\Sigma_\beta = \E[D_\beta D_\beta \tp] $. Note $\E[D_\beta] = 0$.

    One could further write out the expression. Assume $I = \{1,\dots, k\}$.
    Let $\Omega \defeq \cov[\must] = \Diag(\var[\xsti \vi])$.  
    Then the matrix $ (\cP \cH \cP)\pinv= \Diag( 0_{k\times k}, (\cH_{\Ic\Ic})\inv )$ and 
    $  \var[\xst(\theta) \odot v(\theta) - \mubarst]= \E[(\xst(\theta) \odot v(\theta) - \mubarst)^{\otimes 2}] =\Diag ( (\Omega^*_i )\sq )$ where $(\Omega^*_i )\sq = \int (\xsti \vi) \sq \diff S - (\int \xsti \vi \diff S)\sq$ is the variance of winned item value of buyer $i$. 

    \emph{Proof of $\betagam \toprob \betast$.} This follows from \cref{thm:the_shapiro_thm}. 

    \emph{Proof of CLT for pacing multiplier $\beta$.} This follows from the above discussion.

    \emph{Proof of CLT for revenue $\REV$.} 
    We use a stochastic equicontinuity argument.
    Given the item sequence $\gam = (\theta^1, \theta^2 \dots )$, define the (random) operator 
    \begin{align*}
    \nu_t g = \sqrt t (P_t - P)g = \frac{1}{\sqrt t} \sumtau ( g(\thetau) - \E[g]) 
    \;.
    \end{align*}
    where $g:\Theta \to \R$, $ \E[g]=\int g \diff S$.
    Note $\pst(\theta) = \max_i \betasti \vithe = f(\theta,\betast)$, $\REV^* = P f(\cdot ,\betast)$, 
    $\pgam(\thetau) = f(\thetau, \betagam)$ and 
    $\REV^\gam = P_t f(\cdot, \betagam)$ we obtain the decomposition
    \begin{align*}
     \sqrt t (\REV^\gam -  \REV^* ) =
     \underbrace{ \frac{1}{\sqrt t} \sumtau \big( f(\thetau, \betast) - \fbar (\betast) \big) }_{ = : \I_t}
     + 
     \underbrace{    \nu_t (f(\cdot , \betagam  ) - f( \cdot,\betast ))
     }_{ = : \II_t }
     \\
     + 
     \underbrace{ \sqrt{t}( \bar{f} (\betagam) - \bar{f} (\betast) )}_{=: \III_t}
    \end{align*}
    
    For the term $\I_t$, it can be written as
$
        \I_t = \nu_t ( \pst(\cdot) - \REV^* )
$.
    By the linear representation for $\betagam - \betast$ in \cref{eq:expansion_beta}, applying the delta method, we get the linear representation result 
    \begin{align*}
        \III_t  = \frac{1}{\sqrt t} \sumtau \nabla \fbar(\betast) \tp D_\beta(\thetau) + o_p(1)
        =  \frac{1}{\sqrt t} \sumtau (\mubarst)  \tp D_\beta(\thetau) + o_p(1) 
    \end{align*}

    We will show $\II_t= o_p(1)$. The difficulty lies in that the operator $\nu_t$ and the pacing multiplier $\betagam$ depend on the same batch of items. This can be handled with the stochastic equicontinuity argument. The desired claim $\II_t = o_p(1)$ follows by verifying that the function class $\cF_1 = \{ \theta \mapsto  f(\cdot, \beta ) :\beta \in B \}$ (same as that defined in \cref{lm:verifying_VC}) is  VC-subgraph and Donsker. This is true by \cref{lm:verifying_VC}.
    By \cref{lm:donsker_to_SE} we know for any $\delta_t \downarrow 0$,
    \begin{align}
        \sup_{w \in [\delta_t]_{L^2}} \nu_t ( f(\cdot, w) - f(\cdot, \betast) ) = o_p(1)
        \label{eq:f_donsker}
    \end{align}
    where $[\delta_t]_{L^2} = \{ \beta :\beta  \in B, \int (f(\cdot, \beta ) - f(\cdot, \betast))\sq \diff S \leq \delta_t\}$. Noting that for all $\beta, w, \theta$, it holds $|f(\theta, \beta) - f(\theta,w)| \leq \vbar \|\beta -w\|_\infty$, we know that $\int\left[f(\cdot, \beta)-f\left(\cdot, \betast \right)\right]^2 \diff S \rightarrow 0 $ as $ \beta \rightarrow \betast$. Then by \cref{lm:ltwo_to_parameter}, we know \cref{eq:f_donsker} holds with $[\delta_t]_{L^2} $ replaced with $[\delta_t] = \{ \beta: \beta \in B,  \|\beta  - \betast\|_2 \leq \delta_t\} 
    $. Combined with the fact that $\betagam \toprob \betast$, by \cref{lm:se} we know $\II_t = o_p(1)$.
    
    To summarize, we obtain the linear expansion
    \begin{align}
        \label{eq:expansion_rev}
  \sqrt t (\REV^\gam -  \REV^* ) 
 =
    \frac{1}{\sqrt t} \sumtau ( 
    \pst(\thetau) - \REV^*
    + (\mubarst) \tp D_\beta(\thetau)
    ) + o_p(1) .
    \end{align}

We complete the proof of \cref{thm:clt}.

\end{proof}

\subsection{Proof of \cref{cor:clt_u_and_nsw}}  \label{sec:full:cor:clt_u_and_nsw}

\emph{Full statement of \cref{cor:clt_u_and_nsw}.}     Under the same conditions as \cref{thm:clt},
       $\sqrt{t} (\ugam - \ust) $, $\sqrt t(\deltagam - \deltast)$ and $\sqrt{t}(\NSW^\gam - \NSW^*) $
       are asymptotically normal with 
(co)variances
$\Sigma_u \defeq\allowbreak \Diag{({b_i}/({\betasti})\sq)} \allowbreak \Sigma_\beta  \allowbreak    \Diag({b_i}/({\betasti})\sq)$, 
$\Sigma_\delta \defeq \allowbreak  (I_n - \cH(\cH_B)\pinv) \allowbreak \Omega \allowbreak (I_n - \cH(\cH_B)\pinv) \tp$, and
$\sigma^2_\NSW \defeq\Vec(b_i/\betasti)\tp \allowbreak \Sigma_\beta \allowbreak\Vec(b_i / \betasti)$, respectively.

\begin{proof}[Proof of \cref{cor:clt_u_and_nsw}]

    \emph{Proof of CLT for individual utility $u$.} We use the delta method; see Theorem 3.1 from \citet{van2000asymptotic}. Note $ \ust = g(\betast)$ with $g: \Rn \to \Rn, g(\beta) = [b_1/\beta_1,\dots, b_n / \beta_n] \tp$. By \cref{eq:expansion_beta}, it holds 
    \begin{align}
        \label{eq:expansion_u}
        \sqrt{t} (\ugam -\ust ) =  \frac{1}{\sqrt t} \sumtau \nabla g(\betast) D_\beta (\thetau) + o_p(1) .
    \end{align}
    Finally, noting $\nabla g(\betast) = \Diag(-b_i / (\betasti)\sq )$ we complete the proof.

    \emph{Proof of CLT for Nash social welfare $\NSW$.} We use the delta method. Note $ \NSW^* = g(\betast)$ with $g: \Rn \to \R, g(\beta) =  \sumiton b_i \log(b_i / \beta_i) $. By \cref{eq:expansion_beta} it holds 
    \begin{align}
        \label{eq:expansion_nsw}
        \sqrt{t} (\NSW^\gam - \NSW^*) =  \frac{1}{\sqrt t} \sumtau \nabla g(\betast) \tp D_\beta (\thetau) + o_p(1) . 
    \end{align}
    Finally, noting $\nabla g(\betast) = \Vec(b_i / \betasti)$ we complete the proof of \cref{cor:clt_u_and_nsw}.

\emph{Proof of CLT for leftover budget $\delta$.}
This is a direct consequence of Theorem 4.1 in \citet{shapiro1989asymptotic}. By that theorem, it holds that 
\begin{align*}
    \sqrt{t} 
    \begin{bmatrix}
        \betagam - \betast 
        \\
        \delta^\gam_I - \delta^*_I 
    \end{bmatrix}
    \tod 
    \cN(0, 
    \Sigma_{\text{joint}})
\end{align*} 
with 
\begin{align*}
    \Sigma_{\text{joint}} = 
    \begin{bmatrix}
        \cH & A  \tp
        \\
        A &  0
    \end{bmatrix} \inv
    \begin{bmatrix}
       \Omega & 0
        \\
        0 &    0     \end{bmatrix}
    \begin{bmatrix}
        \cH & A \tp
        \\
        A & 0
    \end{bmatrix} \inv
    = \begin{bmatrix}
        (\cH_B)\pinv  \Omega (\cH_B)\pinv &   
        [Q \Omega (\cH_B)\pinv  ]\tp 
        \\
        Q \Omega (\cH_B)\pinv  
        & 
        Q \Omega Q\tp  
    \end{bmatrix}
\end{align*}
where $Q =  (A\cH\inv A\tp)\inv A \cH \inv \in \R^{|I|\times n}$ and $\Omega = \cov[\xst \odot v - \mubarst]$.
By a matrix equality, noting matrix $A$'s rows are distinct basis vectors, it holds 
\begin{align*}
    (A\cH\inv A\tp)\inv A \cH \inv  = A (I_n - \cH (\cH_B)\pinv)
\end{align*}
Moreover, for other entries of $\deltagam$, i.e., $\deltagam_{\Ic}\,$, their asymptotic variance will be zero.
The matrix $(I_n - \cH (\cH_B)\pinv) \Omega (I_n - \cH (\cH_B)\pinv)\tp$ is zero at the $(i,j)$-th entry if $i$ or $j\in \Ic$.
Summarizing, the asymptotic variance of $\sqrt t(\deltagam - \deltast)$ is $(I_n - \cH (\cH_B)\pinv) \Omega (I_n - \cH (\cH_B)\pinv)\tp$.

An alternative proof is by the delta method and a stochastic equicontinuity argument. It holds $\sqrt t ({\deltagam - \deltast}) \tod \cN(0, (I_n - \cH(\cH_B)\pinv) \Omega (I_n - \cH(\cH_B)\pinv) \tp )$ and the linear expansion 
\begin{align}
    \sqrt t (\deltagam - \deltast) = \invtroot \sumtau (I_n - \cH (\cH_B)\pinv) (\must (\thetau)- \mubarst) + o_p(1)
\end{align}
holds. In the case where $I = \emptyset$, i.e., $\deltasti = 0$ for all $i$, we have $\cH_B =\cH$ and so $I_n - \cH (\cH_B)\pinv = 0$.
\end{proof}

\subsection{Proof of \cref{thm:rev_localopt}}
\label{sec:proof:rev_local_asym_risk}

\begin{proof}[Proof of \cref{thm:rev_localopt}]
    Based on Le Cam's local asymptotic normality theory \citep{le2000asymptotics}, 
    to establish the local asymptotic minimax optimality of a statistical procedure,
    one needs to verify two things.
    First, the class of perturbed distributions (the class $\{ s_{\alpha,g} \}_{\alpha,g}$ in our case) satisfies the locally asymptotically normal (LAN) condition \citep{vaart1996weak,le2000asymptotics}. This part is completed by Lemma 8.3 from \citet{duchi2021asymptotic} since our construction of perturbed supply distributions follows theirs.
    Second, one should verify the asymptotic variance of the statistical procedure equals to the minimax optimal variance. To calculate this quantity, one needs to find the derivative of the map $\alpha \mapsto \REVst_{\alpha,g}$ at $\alpha=0$. 
    For this part we present a formal derivation below.

    For a given perturbation ${(\alpha,g)}$, we let $\pst_{\alpha,g}$ and $\REVst_{\alpha,g}$ be the limit FPPE price and revenue under supply distribution $s_{\alpha,g}$.
    Let $S_{\alpha,g}(\theta) = \nabla _\alpha \log s_{\alpha,g}(\theta)$ be the score function. Obviously with our parametrization of $s_{\alpha,g}$ we have $S_{0,g}(\theta) = g(\theta)$ by \cref{eq:perturbed_is_roughly_expo}.
    We next find the derivative of $\alpha\mapsto  \REVst_{\alpha,g} $ at $\alpha=0$.
    Recall $f$ is defined in \cref{eq:pop_deg} and the price is produced by the highest bid, i.e., $\pst_{\alpha,g}(\theta) = \max_i \betast_{\alpha,g}\vithe = f(\theta,\betast_{\alpha,g})$.
    \begin{align}
        & \nabla_ \alpha \REVst_{\alpha,g} = \nabla_ \alpha \int f(\theta,\betast_{\alpha,g}) s_{\alpha,g}(\theta)\diff \theta
        \notag 
        \\
        & = \int [\nabla_\beta f(\theta, \betast_{\alpha,g})\nabla_ \alpha \betast_{\alpha,g} + f(\theta,\betast_{\alpha,g}) S_{\alpha,g}(\theta)] s_{\alpha,g}(\theta)\diff \theta
        \notag 
        \; .
    \end{align}
    Above we exchange the gradient and the expectation and then apply the chain rule.
    By a perturbation result by Lemma 8.1 and Prop.\ 1 from \citet{duchi2021asymptotic}, under \nameref{as:smoothness} and \nameref{as:constraint_qualification}, 
    \begin{align*}
        \nabla_\alpha \betast_{\alpha,g}|_{\alpha = 0} = - (\cH_B)\pinv \Sigma_{\must , g } 
    \end{align*} with 
    $\Sigma_{\must , g }  = \E_s[(\must(\theta) - \mubarst)g(\theta)\tp] $.
    Plugging in $\E_s[\nabla_\beta f(\theta,\betast_{0,g})] = \mubarst$, $f(\theta,\betast_{0,g}) = \pst(\theta)$ (see \cref{sec:fppe_properties})  and $S_{0,g} = g$,
    we obtain 
    \begin{align*}
        \nabla_\alpha\REVst _{\alpha,g} |_{\alpha= 0} 
        & =-(\mubarst)\tp (\cH_B)\pinv \Sigma_{\must, g}  + \E_s[ (\pst (\theta) - \REVst) g(\theta) ] 
        \\
        & = \E\big[ \big( -(\mubarst)\tp (\cH_B) \pinv (\must(\theta) - \mubarst) + (\pst(\theta) - \REVst)\big) g (\theta)\big] 
        \\
        & = \E[ D_\REV(\theta) g(\theta)]
        . 
    \end{align*}
    Now we have the two components required to invoke the local minimax result. 
    Given a symmetric quasi-convex loss $L:\R\to\R$, we recall the local asymptotic risk for any procedure $\{\hat r _t : \Theta^t \to \R\}$ that aims to estimate the revenue:
    \begin{align*}
       & \LAR_\REV ( \{\hat r_t \}) =
        \\ 
        & 
        \sup_{ g\in G_d, d\in \mathbb{N}}
        \lim_{c\to \infty}
        \liminf_{t \to \infty}
        \sup_{\|\alpha\|_2\leq \frac{c}{\sqrt t}}
        \E_{s_{\alpha,g} ^{\otimes t}}[L(\sqrt{t} (\hat r_t  - \REVst_{\alpha,g} ))] \;. 
    \end{align*} 
    Following the argument in \citet[Sec.\ 8.3]{duchi2021asymptotic} it holds 
    \begin{align*}
       \inf_{ \{\hat r_t\}} \LAR_\REV ( \{\hat r_t \}) \geq \E[L(\cN(0, \E[D_\REV\sq(\theta)]))] \;.
    \end{align*}
    We complete the proof of \cref{thm:rev_localopt}.
    
\end{proof}

\subsection{Proof of \cref{thm:variance_estimation}}
\label{sec:proof:thm:variance_estimation}
\begin{proof}[Proof of \cref{thm:variance_estimation}]
    We first show $\hat \cH \toprob \cH$ by verifying conditions in \cref{lm:hessian_estiamtion}.
    All conditions are easy to verify except the stochastic equicontinuity condition. By \cref{lm:verifying_VC} we know the SE condition holds. We conclude $\hat \cH \toprob \cH$.
    
    Next we show $\P(\hat \cP = \cP ) \to 1$. Recall $\cP = \Diag( \indi( \betasti < 1))$ indicates the set of 
    inactive constraints. For $i$ such that $\betasti = 1$, we know $\betagami - 1 = O_p(\frac{1}{\sqrt t})$ by the central limit theorem results \cref{thm:clt} (actually the strong result $\betagami - 1 = o_p(\frac{1}{\sqrt t})$ holds). Then 
    \begin{align}
        \P( \betagami < 1-\varepsilon_t ) = \P( O_p(1) > \sqrt{t} \varepsilon_t)  \to 0,
        \tag{by $\varepsilon_t\sqrt{t} \to \infty$}
    \end{align}
    using the smoothing rate condition $ \sqrt{t} \varepsilon_t\to \infty$.
    For $i$ such that $\betasti < 1$, we know $\betagam - \betasti = o_p(1)$ by consistency of $\betagam$. Then 
    \begin{align}
        \P( \betagami < 1-\varepsilon_t ) = \P (o_p(1) < (1-\betasti)  - \varepsilon_t) \to 1.
        \tag{by $\varepsilon_t = o(1)$ and $1-\betasti > 0$}
    \end{align}
    We conclude $\P(\hat\cP = \cP) \to 1$.
    
    We now show $(\hat \cP \hat \cH \hat \cP)\pinv  \toprob ( \cP\cH \cP)\pinv $. For any $\epsilon > 0$,
    \begin{align*}
        & \P( \| ( \hat \cP \hat \cH \hat \cP)\pinv  - ( \cP\cH \cP)\pinv \|_F > \eps) 
        \\
        & \leq \P( \|( \hat \cP \hat \cH \hat \cP)\pinv  - ( \cP\cH \cP)\pinv \|_F > \eps, \hat \cP = \cP) + \P(\hat \cP\neq \cP)  
        \\
        & = \P( \| [\hat \cH_{\Ic \Ic} ] \inv  - [\cH_{\Ic \Ic}]\inv \|_F > \eps) + \P(\hat \cP\neq \cP)  
        \to 0 \tag{by $\hat\cH \toprob \cH$}
    \end{align*}
    
    Next we show $\Omega = \E[(\xst\odot v - \mubarst)^{\otimes 2} ]$ can be consistently estimated by 
    \begin{align*}
       \hat \Omega = \frac1t\sumtau ( \xgam(\thetau)\odot v(\thetau) - \mubargam)^{\otimes 2} 
    \end{align*}
    Define for $\beta\in N_{\subdiff}$
    \begin{align*}
        \hat \Omega (\beta) &=  \frac{1}{t} \sumtau \Big(d(\thetau, \beta) -  \frac{1}{t} \sum_{s=1}^t d(\theta^s,\beta) \Big)^{\otimes 2} \\
        \Omega (\beta) &  =  \E[ (\nabla f(\theta, \beta) - \nabla \fbar (\beta)) ^{\otimes 2}] 
    \end{align*}
    where $$d(\theta, \beta) \in \partial f(\theta, \beta)= \text{convexhull}\{e_iv_i: i \text{ such that }\betai v_i(\theta) = \max_k \beta_k v_k(\theta)\}$$ is a fixed selection of an element from the subgradient set, $N_{\subdiff}$ is a neighborhood of $\betast$ such that for all $\beta\in N$, the set $\{ \theta: f(\theta, \cdot) \text{ not differentiable at $\beta$}\}$ is measure zero; see \cref{lm:hessian_implications}.
    \footnote{
        Note this is different from the condition there is a neighborhood $N$ such that the set $\{\theta: f(\theta, \cdot) \text{ is differentiable on $N$} \}$ is measure one. 
        The map $(\theta,\beta) \mapsto \nabla f(\theta,\beta)$ is defined on $(\bigcap_{\beta\in N} \Thetatie(\beta)) \times N$, and yet the set $\bigcap_{\beta\in N} \Thetatie(\beta) $ might not be measure one.
        This is the reason we use subgradients in the definition of $\hat \Omega(\cdot)$.}
    We also assume the subgradient selection coincides with the observed FPPE allocation selection, i.e., $$d(\thetau, \betagam) =\xgam(\thetau)\odot v(\thetau) \in \partial f(\thetau ,\betagam) .$$
    Noting 
    $\E[\xst\odot v ] = \nabla \fbar (\betast) = \mubarst $ and 
    $\frac1t\sumtau d(\thetau, \betagam) = \frac1t \sumtau \xgam(\thetau)\odot v(\thetau) =  \mubargam$, we see that $\hat \Omega(\cdot)$ and $\Omega(\cdot)$ are constructed so that $\hat \Omega(\betagam) = \hat \Omega$ and $\Omega(\betast) = \Omega$.
    Moreover, for any fixed $\beta \in \Ndiff$, it holds $\E[d(\theta,\beta)] = \nabla \fbar (\beta)$ \citep[Prop.\ 2.2]{bertsekas1973stochastic}.
    Next we use a decomposition of $\hat \Omega (\beta)$
    \begin{align*}
        \hat \Omega (\beta) = 
        \underbrace{\frac1t \sumtau \Big(d(\thetau,\beta) - \nabla \fbar (\beta)\Big)\ot }_
        {\I(\beta)}
        - 
        \underbrace{
        \bigg(\frac1t \sumtau d (\thetau,\beta) - \nabla \fbar(\beta) \bigg)\ot}
        _
        {\II(\beta)}
    \end{align*}
    We now argue that both terms converges in probability uniformly on $N_{\subdiff}$. For the first term,
    consider a fixed $\beta\in \Ndiff$, we know
    \begin{align*}
        \{\theta: d(\theta,\cdot ) \text{ not continuous at $\beta$} \} 
        = \Thetatie(\beta)
    \end{align*}
    which is measure zero, and that $\nabla \fbar$ is continuous on $\Ndiff$ by \nameref{as:smoothness}.
    By Theorem 7.53 of~\citet{shapiro2021lectures} (a uniform law of large number result for continuous random functions), it holds 
    $ \sup_{\beta \in N_{\subdiff}} \| \frac1t \sumtau (d(\thetau,\beta) - \nabla \fbar(\beta))\ot  - \Omega(\beta)\| \toprob 0$
    and $\sup_{\beta \in N_{\subdiff}}\II(\beta)\toprob 0$. 
    To summarize we have 
    \begin{align*}
        \sup_{\beta \in N_\subdiff}\|  \hat \Omega (\beta) -  \Omega (\beta)\| \toprob 0.
    \end{align*}
    Combined with the fact that $\betagam \toprob \betast$ we have $\hat \Omega(\betagam) \toprob \Omega(\betast)$. From this we conclude $\hat \Omega \toprob \Omega$.
    
    \emph{Proof of $\hat\Sigma_\beta \toprob \Sigma_\beta$.} We rewrite $\hat\Sigma_\beta$ as 
    \begin{align*}
        \hat\Sigma_\beta &=  (\hat \cP \hat \cH \hat \cP)\pinv \bigg(\frac1t \sumtau (\xgam(\thetau) \odot v(\thetau) - \mubargam) \ot \bigg) (\hat \cP \hat \cH \hat \cP)\pinv
        \\
        & =  (\hat \cP \hat \cH \hat \cP)\pinv \hat\Omega  (\hat \cP \hat \cH \hat \cP)\pinv 
        \\ 
        & \toprob 
        ( \cP  \cH  \cP)\pinv \Omega  ( \cP  \cH  \cP)\pinv = \Sigma_\beta
    \end{align*}
    
    \emph{Proof of $\hat\sigma\sq_\REV \toprob \sigma\sq_\REV$.} 
    In \cref{lm:verifying_VC} we have shown both $\cF_1$ and $\cF_2$ are VC-subgraph, and thus a uniform law of large number holds.
    We rewrite 
    \begin{align*}
        \sigma\sq_\REV = 
        \underbrace{\E[(p\st - \REV\st)\sq ]}_{\I_t} + 
        \underbrace{(\mubarst)\tp  ( \cP  \cH  \cP)\pinv \Omega  ( \cP  \cH  \cP)\pinv \mubarst }_{\II_t}
        \\
        + 
        \underbrace{2 \E[ (p\st - \REV\st)(\xst\odot v - \mubarst)] \tp ( \cP  \cH  \cP)\pinv \mubarst }_{\III_t} 
    \end{align*}
    and 
    \begin{align*}
    \hat \sigma\sq_\REV  = 
    \underbrace{\frac1t\sumtau (\pgam(\thetau) - \REV^\gam)\sq}_{\hat\I_t} + 
    \underbrace{(\mubargam)\tp (\hat \cP \hat \cH \hat \cP)\pinv \hat\Omega  (\hat \cP \hat \cH \hat \cP)\pinv \mubargam }_{\hat \II_t}
    \\
     +
    \underbrace{2 \bigg(\frac1t \sumtau 
    (\pgam(\thetau) - \REV^\gam) (\xgam(\thetau) \odot v(\thetau) - \mubargam)\bigg) \tp (\hat \cP \hat \cH \hat \cP)\pinv  \mubargam }_{\hat \III_t}
    \end{align*}
    We have $\hat\I_t\toprob \I_t$ by invoking \cref{lm:verifying_VC}, applying a uniform LLN and using the fact that $\betagam \toprob \betast$. And $\hat \II_t\toprob \II$ holds by 
    $\mubargam \toprob \mubarst$, $(\hat \cP \hat \cH \hat \cP)\pinv  \toprob ( \cP\cH \cP)\pinv$ and $\hat \Omega \toprob \Omega$, and applying 
    Slutsky's theorem.
    Finally, $\hat\III_t\toprob \III$ by $\cF_1 \cdot \cF_2$ is Donsker by \cref{lm:donsker_multiplication} and thus a uniform law of large number holds, and that $\betagam\toprob\betast$. 

    We complete the proof of \cref{thm:variance_estimation}.
    \end{proof}

\subsection{Proof of \cref{thm:clt_ab_testing}}\label{sec:proof:thm:clt_ab_testing}
\begin{proof}[Proof of \cref{thm:clt_ab_testing}]
By the EG characterization of FPPE, we know that $\betagam(1)$, the pacing multiplier of the observed FPPE $\oFPPE \big(\pi b, v(1), \frac{\pi}{t_1}, \gamma(1)  \big)
$, solves the following dual EG program
\begin{align}
    \min_{B} \frac{1}{t_1} \sum_{\tau = 1}^{t_1} \max_i \vithetau \betai - \sumiton b_i \log(\betai)
\end{align}

The major technical challenge is that the number of summands in the first summation is also random. 
Given a fixed integer $k$ and a sequence of items $( \theta^{1,1}, \dots, \theta^{1,k})$, define
    \begin{align*}
        & \beta^{\lin, k} (1) = \betast(1) +  \frac1{{k}} \sum_{\tau = 1}^{k} D_\beta(1, \theta^{1,\tau}), 
        \\ 
        & \beta^k (1)= \text{the unique pacing multiplier in $\oFPPE(b,v(1), k\inv, ( \theta^{1,1}, \dots, \theta^{1,k}))$}
    \end{align*}
    Here $D_\beta(1,\cdot) =  - (\cH_B(1))\pinv (\must(1,\cdot) - \mubarst(1))$ 
    where
$\cH_B(1), \must(1, \cdot)$ and $\mubarst(1)$ are the projected Hessian in \cref{def:cHB}, item utility function in \cref{def:Omega}, and total item utility vector in \cref{eq:def:must} in the limit market $\FPPE(b, v(1), s)$.
Note $\E[D_\beta(1,\cdot)] = 0$.
Note $\betagam(1) = \beta^{t_1}$ since scaling the supply and the budget at the same time does not change the equilibrium pacing multiplier.
We introduce the following asymptotic equivalence results:
\begin{lemma} \label{lm:asym_equivalence_beta}
     Recall $t_1 \sim \text{Bin}(\pi, t)$. If \nameref{as:constraint_qualification} and \nameref{as:smoothness} hold for the limit market $\FPPE(b,v(1),s)$, then
     \begin{itemize}
         \item $ \sqrt{t} (\betagam(1) - \beta^{\lin, t_1}) = o_p(1)$ as $t \to \infty$.
         \item $ \sqrt{t} (\beta^{\lin, t_1} - \beta^{\lin, \lfloor \pi t \rfloor}) = o_p(1) $ as $t \to \infty$.
     \end{itemize}
  
     Here $\lfloor a \rfloor $ is the greatest integer less than or equal to $a \in \R$. 
     A similar result holds for the market limit $\FPPE(b,v(0),s)$ and the influence function $D_\beta(0, \cdot)$ is defined similarly.
\end{lemma}

With \cref{lm:asym_equivalence_beta}, we write
\begin{align*}
 & \sqrt{t} (\hat \tau_\beta - \tau _\beta)
 \\
 & = \sqrt t (\betagam(1) - \betast(1) ) - \sqrt{t} (\betagam(0) - \betast(0))
 \\ 
 & = \sqrt t \bigg( \frac1{\sqrt{\lfloor \pi t \rfloor}} \sum_{\tau = 1}^{\lfloor \pi t \rfloor} D_\beta(1, \theta^{1,\tau})
  -   \frac1{\sqrt{\lfloor (1-\pi) t \rfloor}} \sum_{\tau = 1}^{\lfloor (1-\pi) t \rfloor} D_\beta(0, \theta^{0,\tau})\bigg) + o_p (1) \tag{\cref{lm:asym_equivalence_beta}}
 \\
 & \tod \cN \bigg(0,\frac1{\pi }\var[ D_\beta(1,\cdot)] 
 + \frac1{{(1-\pi) }} \var[ D_\beta(0,\cdot)]\bigg). \tag{independence between $\{\theta^{1,\tau}\}_\tau$ and $\{\theta^{0,\tau}\}_\tau$}
\end{align*}

\emph{Proof of CLT for $\tau_\beta$.} It follows from the above discussion.

\emph{Proof of CLT for $\tau_u$.} We begin with the linear expansion \cref{eq:expansion_u} and repeat the same argument.

\emph{Proof of CLT for $\tau_\REV$.} We begin with the linear expansion \cref{eq:expansion_rev} and repeat the same argument.

\emph{Proof of CLT for $\tau_\NSW$.} We begin with the linear expansion \cref{eq:expansion_nsw} and repeat the same argument.

We complete the proof of \cref{thm:clt_ab_testing}.
\end{proof}

In order to prove \cref{lm:asym_equivalence_beta}, we will need the following lemma.
\begin{lemma} \label{lm:random_index_op}
    If $X_t = o_p(1)$ and $T \sim \text{Bin}(\pi, t)$ and $T$ and the sequence are independent, then $X_T = o_p(1)$.
\end{lemma}
\begin{proof}[Proof of \cref{lm:random_index_op}]
By $X_t = o_p(1)$ we know for all $\eps > 0$ it holds $\P( |X_t| > \eps ) \to 0$, or equivalently $ \sup_{k \geq t} \P( |X_k| > \eps) \to 0$ as $ t \to \infty$. By a concentration for binomial distribution, we know for all $\delta> 0$, it holds $\P( |T - \pi t| > \delta \pi t ) \leq 2\exp(- \delta\sq  \pi t / 3)$. Now write 
\begin{align*}
\P( |X_T| > \epsilon)
&\leq \P( |X_T| > \eps, T\in (1\pm \delta) \pi t ) + \P(T \notin (1\pm \delta) \pi t)
\\
&\leq \sum_{k\in  (1\pm \delta) \pi t } \P(|X_k| > \eps) \P(T = k)+2 \exp(- \delta \sq  \pi t / 3)
\\
&\leq \sup_{k\geq (1-\delta) \pi t} \P(|X_k| > \eps) + 2\exp(- \delta  \sq \pi  t / 3)
\to 0 \text{ as $t \to \infty$}
\end{align*}
where in the second inequality we use the independence between $T$ and the sequence.
We conclude $X_T = o_p(1)$,  completing proof of \cref{lm:random_index_op}.
\end{proof}

\begin{proof}[Proof of \cref{lm:asym_equivalence_beta}]
    The first statement uses the independence between $t_1$ and the items $(\theta^{1,1},\theta^{1,2},\dots)$. 
    Define $R(k) = \sqrt{t} (\beta^k(1) - \beta^{\lin,k}(1))$. By \cref{eq:expansion_beta}, we have $R(k) = o_p(1)$ as $k\to\infty$.
    With this notation, the first statement is equivalent to $R(t_1) = o_p(1)$ where $t_1 \sim \text{Bin}(\pi, t)$, which holds true by \cref{lm:random_index_op}.
    
    The second statement is equivalent to $\sqrt{\lfloor \pi t \rfloor}  \big(\beta^{\lin, t_1} (1)- \beta^{\lin ,\lfloor \pi t \rfloor} (1)\big) = o_p(1)$.
    To prove this we use a Komogorov's inequality.
    By Theorem 2.5.5 from \citet{durrett2019probability}, for any $\eps > 0$, (let $\sigma_{D_\beta} = \E[\|D_\beta(1,\theta)\|_2\sq]^{1/2}$)
    \begin{align*}
        \P \bigg (\sqrt{\lfloor \pi t \rfloor} \sup_{(1-\eps) \lfloor \pi t \rfloor \leq m \leq (1 + \eps) \lfloor \pi t \rfloor} \|\beta^{\lin, m} (1)- \beta^{\lin ,(1-\eps)\lfloor \pi t \rfloor} (1) \| _2 \geq \delta \sigma_{D_\beta}  \bigg) \leq \frac{2 \eps}{\delta\sq}.
    \end{align*}
    Then 
\begin{align*}
    &\P\big( \sqrt{\lfloor \pi t \rfloor} \big\| \beta^{\lin, t_1}(1) - \beta^{\lin ,\lfloor \pi t \rfloor}(1) \big\|_2\geq \delta \big)
    \\
    & \leq \P\big(  \sqrt{\lfloor \pi t \rfloor} \big\| \beta^{\lin , t_1} (1)- \beta^{\lin ,\lfloor \pi t \rfloor} (1) \big\|_2\geq \delta , \; (1 - \eps) \lfloor \pi t \rfloor \leq t_1\leq  (1 + \eps) \lfloor \pi t \rfloor\big)
    \\
    &\quad + 
    \P\big( t_1 \notin \big[(1 - \eps) \lfloor \pi t \rfloor, (1 + \eps) \lfloor \pi t \rfloor\big]\big)
    \\
   & \leq \frac{2\eps \sigma_{D_\beta}\sq}{\delta\sq} + 
    \P\big( t_1 \notin \big[(1 - \eps) \lfloor \pi t \rfloor, (1 + \eps) \lfloor \pi t \rfloor\big] \big) \to\frac{2\eps \sigma_{D_\beta}\sq}{\delta\sq} 
\end{align*}
Finally, since the above holds for all $\epsilon > 0$, we obtain $\sqrt{\lfloor \pi t \rfloor}  (\beta^{\lin, t_1} - \beta^{\lin ,\lfloor \pi t \rfloor}) = o_p(1)$.
    We complete the proof of \cref{lm:asym_equivalence_beta}.
\end{proof}

\section{Experiment Details} \label{sec:exp_details}

\subsection{Constraint identification and fast convergence rate; see \cref{fig:nonnormality_unif,fig:nonnormality_expo,fig:nonnormality_normal}}
\label{sec:exp_fast_rate}

We verify that $\betagami$ converges to $1$ at a faster speed than the usual $O_p(\invtroot)$ if 
$\betasti = 1$, i.e., the constraint is active. We choose the FPPE instances as follows.
\begin{itemize}
    \item $n=25$ buyers and $t = 1000$ items.
    \item budget: $b_i = U_i + 1$ for $i=1,\dots, 5$ and $b_i = U_i$ for $i= 6, \dots, 25$, $U_i$'s are i.i.d.\ uniforms. The extra budgets are to ensure we observe $\betasti = 1$ for the first few buyers.
    \item value and supply: we let $\{ v_1, \cdots, v_n\}$ be identically and independently distributed as either uniforms, exponential, or truncated standard normal.  
\end{itemize}
Under each configuration we form 100 observed FPPEs (100 trials), and plot the histogram of each $\sqrt{t}( \betagami - \betasti)$. 
The population EG \cref{eq:pop_deg} is a constrained stochastic program and can be solved with stochastic gradient based method. In particular, the true value $\betast$ is computed by the dual averaging algorithm \citep{xiao2010dual}. The mean square error decays as $\E[\| \beta^{\text{da}, t} - \betast\|\sq] = O (t \inv) $ with $t$ being the number of iteration, and so if we choose $t$ to be large enough, it will be accurate to replace $\betast$, and we still observe CLT for the quantities $\sqrt t(\beta^{\text{da}, t} - \betagam)$.

We clearly see that (i) if $\betasti < 1$ then the finite sample distribution is close to a normal distribution, and (ii) if $\betasti = 1$ (or very close to $1$, such as $\beta_{14, 21}$ in the uniform value plots, $\beta_{20, 23}$ in exponential), the finite sample distribution puts most of the probability mass at 1.
For cases where $\betasti$ is close to 1, we need to futher increase number of items to observe normality.

\subsection{Choice of smoothing parameter $\varepsilon_t$; see 
\cref{fig:smooth_uniform,fig:smooth_exponential,fig:nonnormality_normal}.}
\label{sec:exp_hessain}
Recall that a key component in the variance estimator is the Hessian estimation, during which 
we choose a smoothing parameter $\varepsilon_t$. In particular, the smoothing $\varepsilon_t$ is used to (1) estimate the active constraints and (2) construct the numerical difference estimator $\hat \cH$. 
\cref{thm:variance_estimation} suggests a choice of 
$\varepsilon_t = t^{-d}$ for some $0 < d < \tfrac12$. We investigate the effect of $d$ in this experiment.

We look at the following configuration of $H_t$ defined in \cref{eq:sample_deg} and the smoothing parameter $\d$.
Note we will be evaluating Hessian at a prespecified point and do not need to form any market equilibria in this experiment.
\begin{itemize}
    \item $n=9$ buyers. The item size $t$ ranges from 200 to 5000, at a log scale. Concretely, we choose $t\in$ [199,  223,  249,  279,  311,  348,  389,  434,  486,  543,  606,
    678,  757,  846,  946, 1057, 1181, 1319, 1474, 1647, 1841, 2057,
   2298, 2568, 2870, 3207, 3583, 4004, 4474, 5000]
    \item budget: it does not play a role in Hessian estimation.
    \item value and supply: uniform, exponential, or truncated standard normal
    \item the smoothing parameter $d \in$ [0.10, 0.17, 0.25, 0.32, 0.40, 0.47, 0.55, 0.62, 0.70].
\end{itemize}
We evaluate the Hessian $\nabla\sq H$ at a pre-specified point $\beta = [0.200, 0.333, 0.467, 0.600, 0.733, 0.867, 1.000]$, and plot the estimated diagonal values, $\hat \cH_{ii}$ for $i \in [7]$, against the number of items $t$.
Under each configuration we repeat for 10 trials. We see that $d$ represents a bias-variance trade-off. For a small $d$ (0.10, 0.17, 0.25), the variance of the estimated value $\hat \cH_{ii}$ is small and yet bias is large (since the plots seem to be trending to some point as number of item increases). For a large $d$ (0.55, 0.62, 0.70) variance is large and yet the bias is small (the estimates are stationary around some point). It is suggested to use $d \in (0.32, 0.47)$.

\subsection{Coverage rate of revenue confidence interval \cref{eq:rev_variance}; see \cref{tbl:rev_coverage,tbl:rev_coverage_40-60,tbl:rev_coverage_100-300}.}
\label{sec:exp_rev_CI}

\begin{itemize}
    \item $n\in $ [10, 20, 30, 40, 50, 60, 100, 200, 300] buyers. Number of items $t\in$ [40, 60, 80, 100, 200, 400].
    \item budget and proportion of buyer with leftover budgets:  $b_i = U_i + 1$ for $i=1,\dots, [\alpha n]$ and $b_i = U_i$ for $i= [\alpha n]+1, \dots, n$, $U_i$'s are i.i.d.\ uniforms. Here $\alpha \in$ [0.4, 0.6, 0.8] is the proportion of buyers with leftover budgets. 
    \item value and supply: uniform, exponential, or truncated standard normal
    \item the smoothing parameter in Hessian estimation $d = 0.4$.
    \item Coverage rate $\alpha = 0.1$, so we construct $90\%$ CIs.
\end{itemize}

Under each configuration, we calculate the true revenue in the limit FPPE with dual averaging and sampling, and then form 100 (or 50 in the larger size experiment) observed FPPEs. For each observed FPPE, we construct the revenue confidence interval as in \cref{eq:rev_variance}. Then we see if the interval covers the true revenue. After 100 trials (or 50 in larger size experiments), we report the coverage rate for that configuration.

We observe that the empirical coverage rate agrees with the nominal coverage rate $90\%$.

\subsection{Effect of treatment assignment probability $\pi$ in A/B testing; see \cref{tbl:abtexting}.}
\label{sec:exp_abtest_CI}
In this experiment we study the effect of treatment assignment probability $\pi$ on the 
coverage rate of treatment effect confidence interval \cref{eq:ab_rev_ci}.

\begin{itemize}
    \item $n\in$ [30, 60, 100] buyers. The item size $t\in$ [100, 200, 400].
    \item budget and proportion of buyer with leftover budgets:  $b_i = U_i + 1$ for $i=1,\dots, [\alpha n]$ and $b_i = U_i$ for $i= [\alpha n]+1, \dots, n$, $U_i$'s are i.i.d.\ uniforms. Here $\alpha = .3 $ is the proportion of buyers with leftover budgets. 
    \item treatment effects on values: suppose before treatment values are uniformly distributed, and after treatment values become exponentially distributed.
    \item the smoothing parameter in Hessian estimation $d =.4$.
    \item the treatment probability $\pi \in$ [.1, .3, .5, .7, .9]
\end{itemize}
We do observe that for markets with fewer buyers (say 30), the treatment probability has an effect on coverage rate if one type of treatment is applied to a small proportion of items; the empirical coverage rates are slightly lower than the nominal $90\%$; see for example the entries corresponding to ``treatment prob = 0.9'' and ``number of buyers = 30''. However, when we increase the number of buyers, the empirical coverage rate agrees with the nominal $90\%$.

\begin{figure}[h!]
    \center
    \includegraphics[scale=.7]{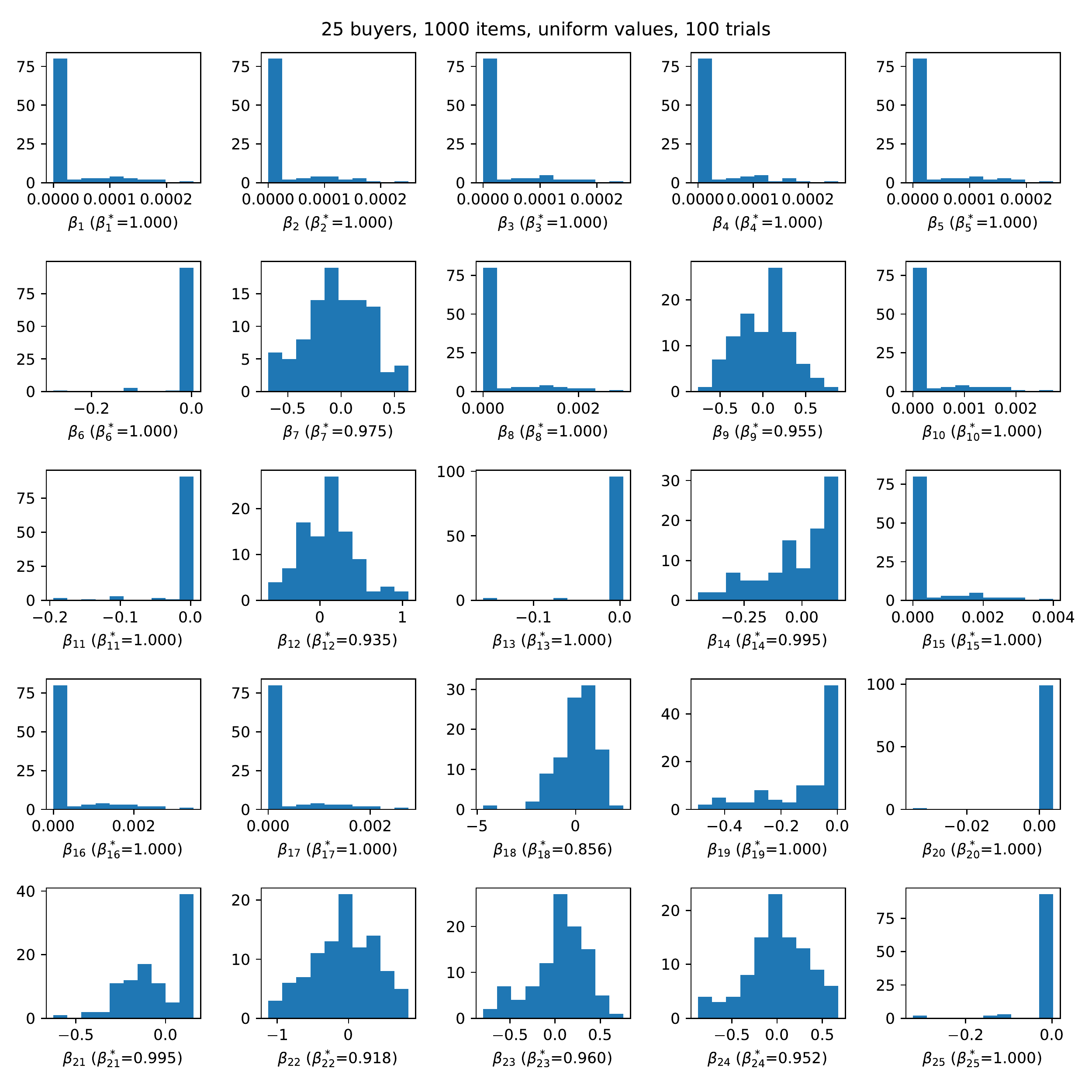}    
    \caption{Distribution of $\sqrt t{(\betagami - \betasti)}$ for all $i \in [n]$.    
    Nonnormality and fast convergence for buyers with $\betasti = 1$. Uniform values.}
    \label{fig:nonnormality_unif}
\end{figure}

\begin{figure}[h!]
    \center
    \includegraphics[scale=.7]{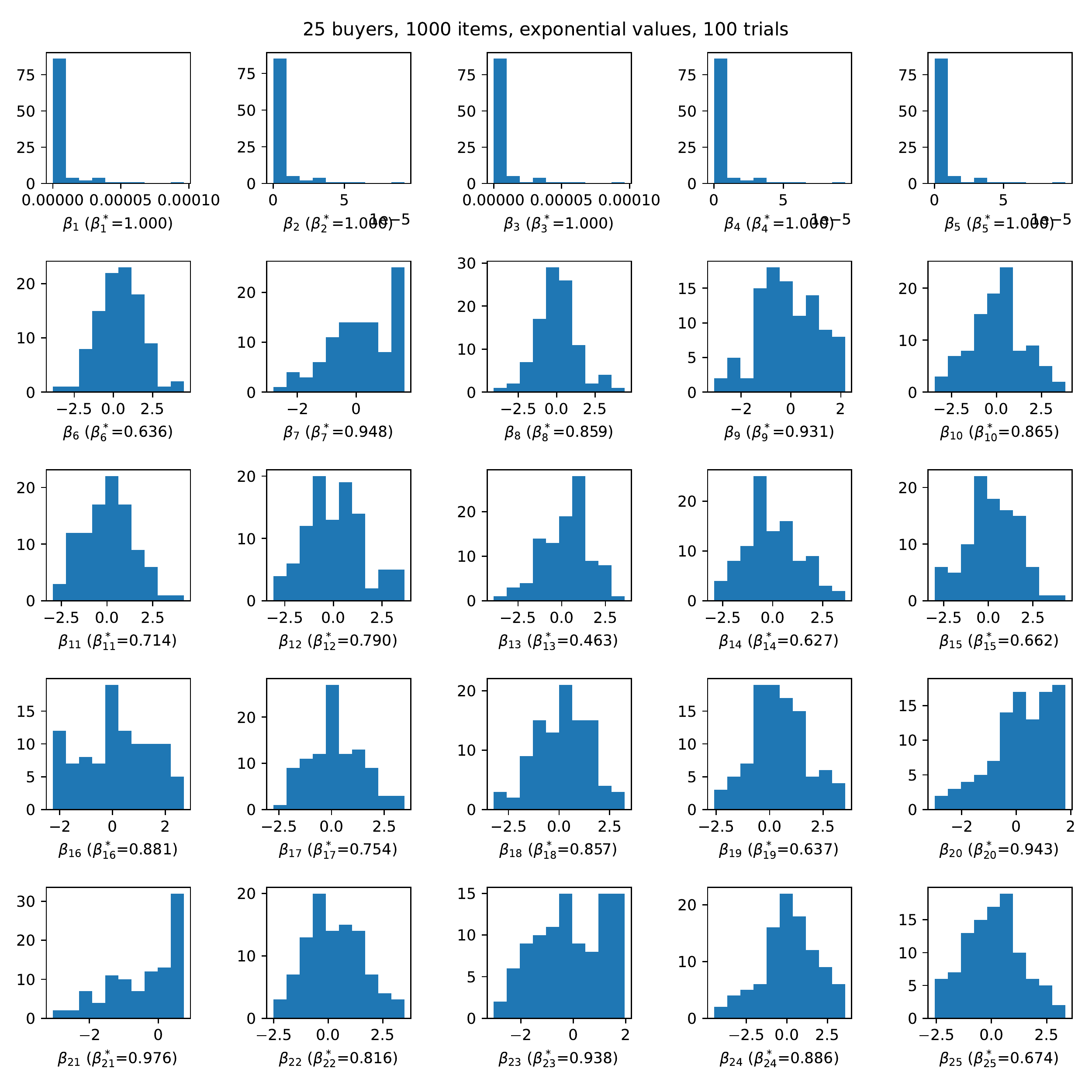}    
    \caption{Distribution of $\sqrt t{(\betagami - \betasti)}$ for all $i \in [n]$.    
    Nonnormality and fast convergence for buyers with $\betasti = 1$. Exponential values.}
    \label{fig:nonnormality_expo}
\end{figure}

\begin{figure}[h!]
    \center
    \includegraphics[scale=.7]{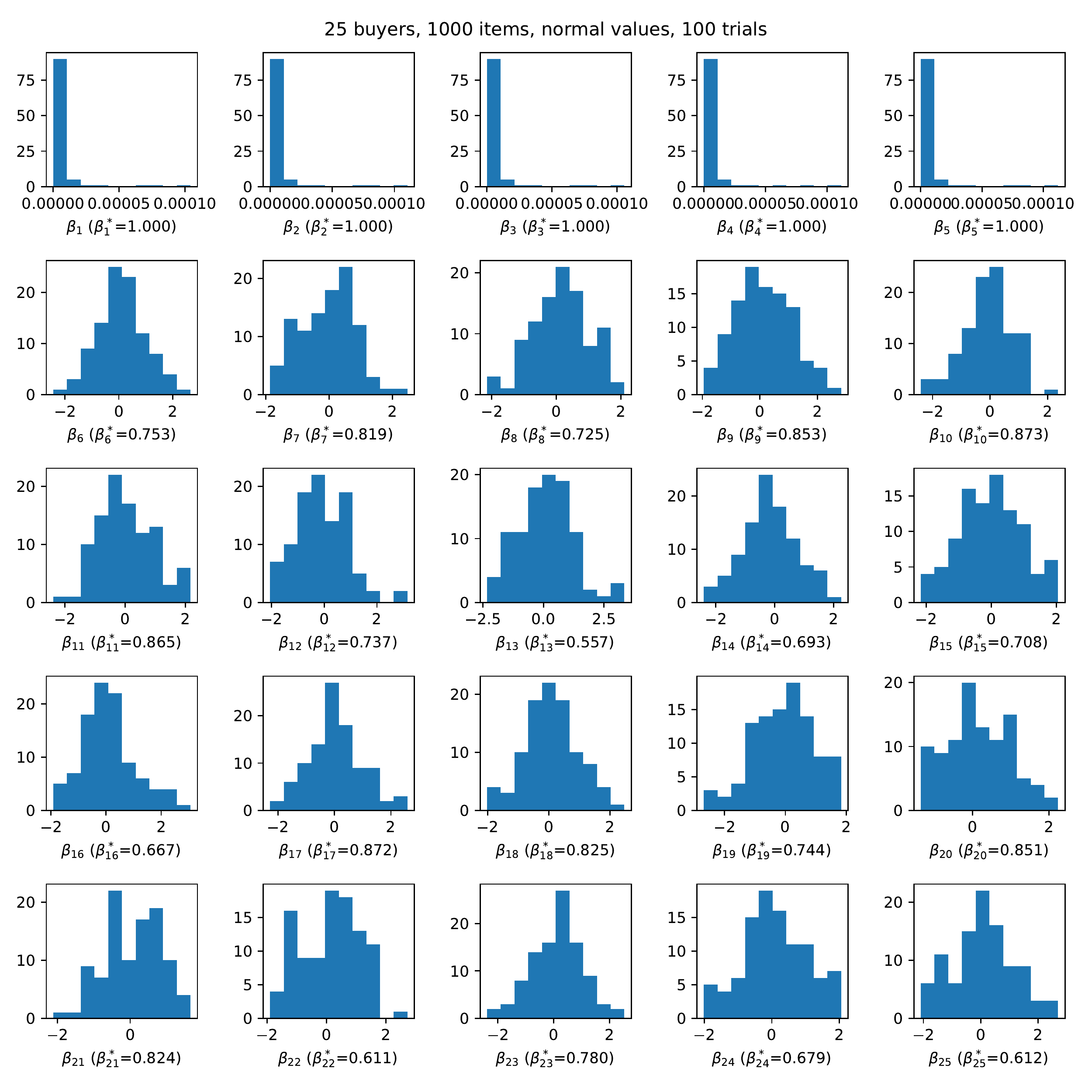}    
    \caption{
    Distribution of $\sqrt t({\betagami - \betasti})$ for all $i \in [n]$.    
    Nonnormality and fast convergence for buyers with $\betasti = 1$. Truncated normal values.}
    \label{fig:nonnormality_normal}
\end{figure}

\begin{figure}[h!]
\center
\includegraphics[scale=.45]{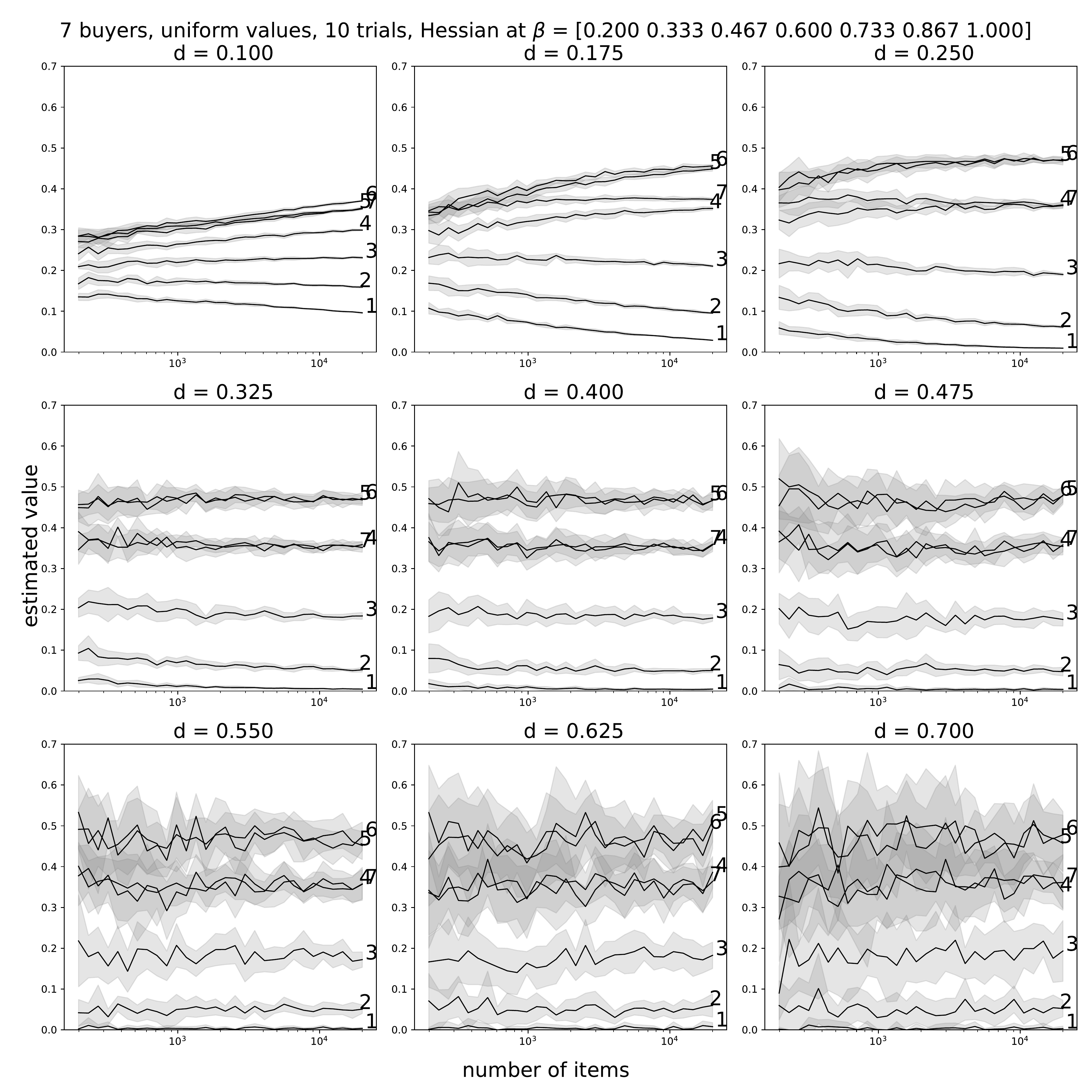}    
\caption{Effect of smoothing parameter on numerical difference estimation of Hessian. Each curve represents the estimated value of $\cH_{ii}$. Uniform values.}
\label{fig:smooth_uniform}
\end{figure}

\begin{figure}[h!]
    \center
    \includegraphics[scale=.45]{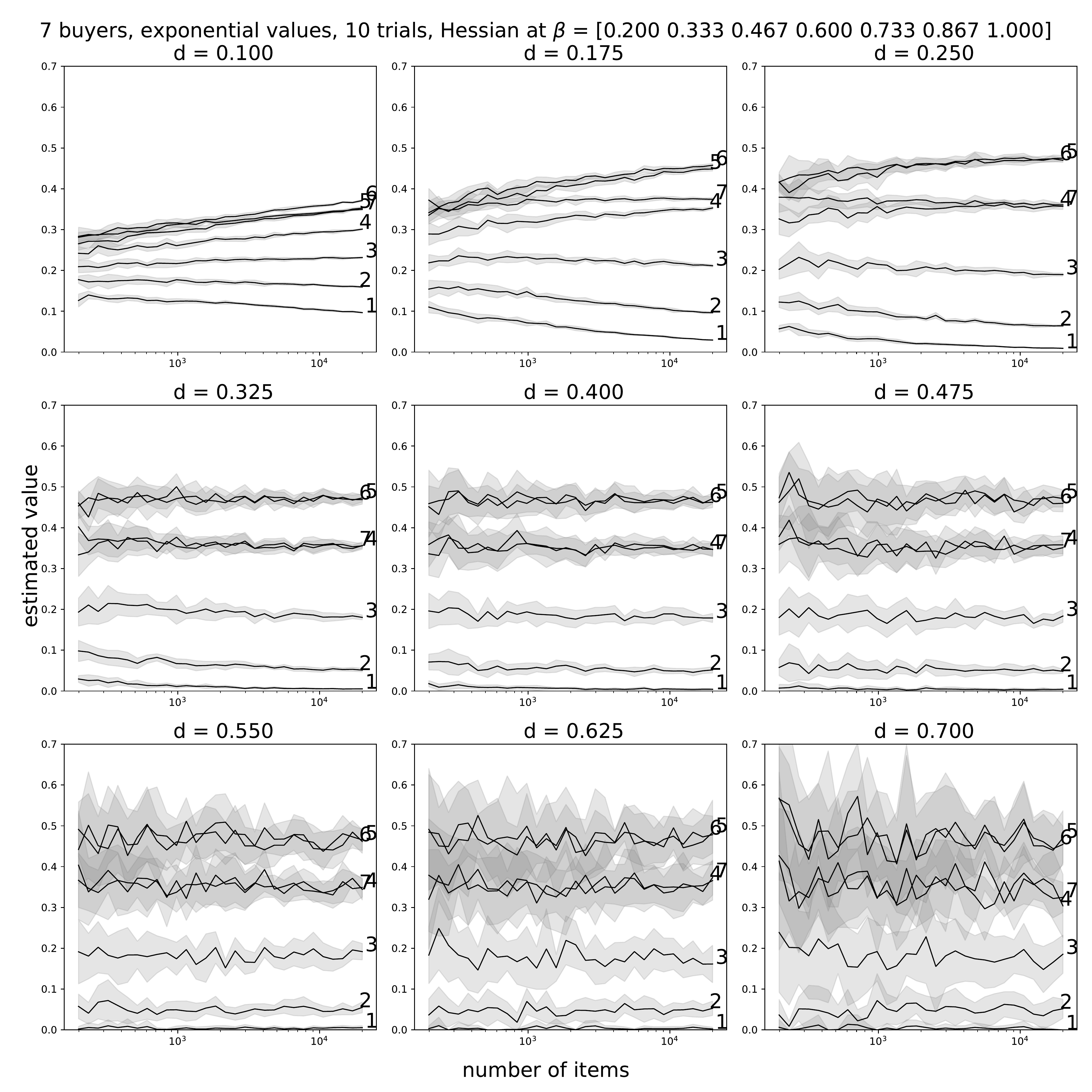}    
    \caption{Effect of smoothing parameter on numerical difference estimation of Hessian. Each curve represents the estimated value of $\cH_{ii}$. Exponential values.}
    \label{fig:smooth_exponential}
\end{figure}

\begin{figure}[h!]
    \center
    \includegraphics[scale=.45]{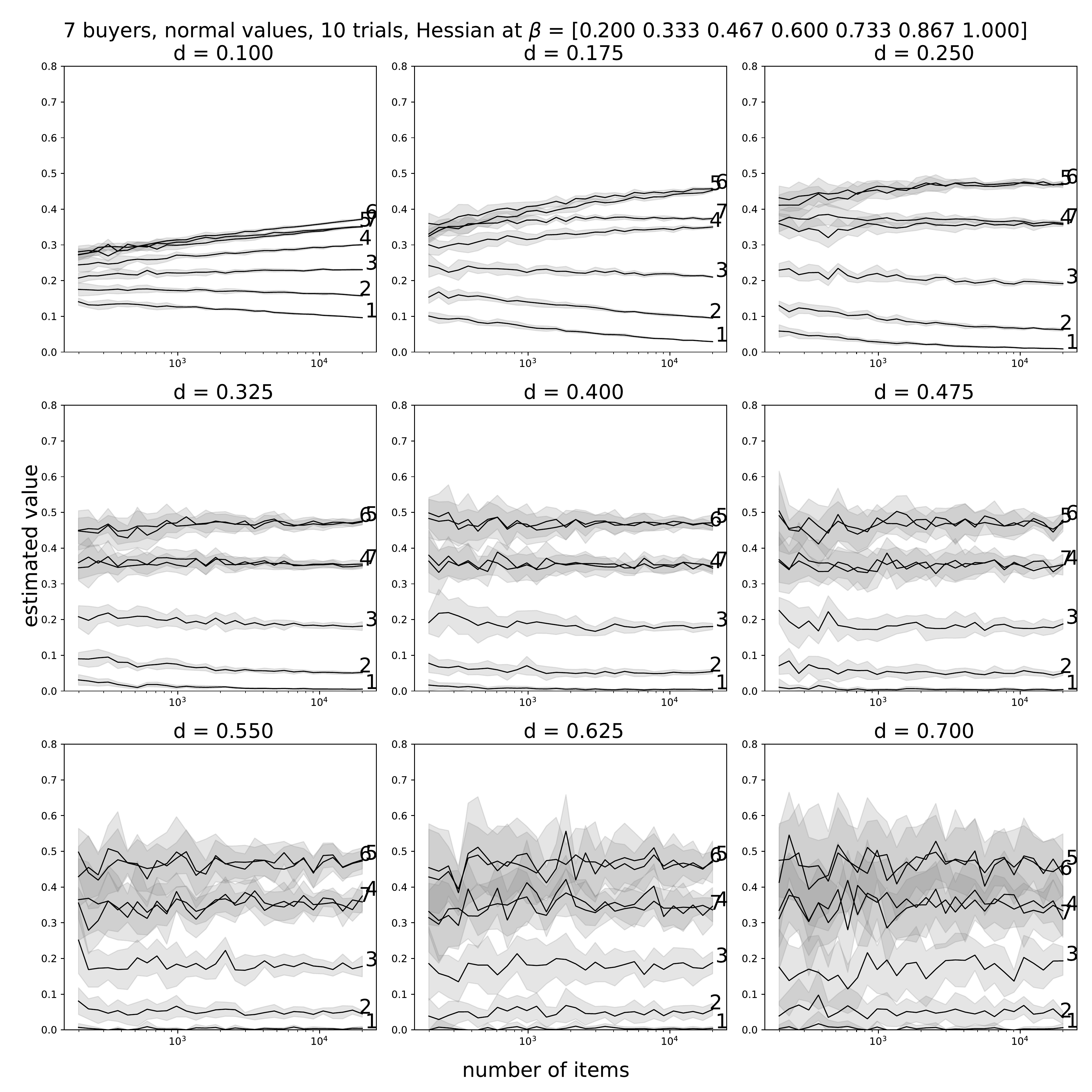}    
    \caption{Effect of smoothing parameter on numerical difference estimation of Hessian. Each curve represents the estimated value of $\cH_{ii}$. Truncated normal values.}
    \label{fig:smooth_normal}
\end{figure}

\begin{table}[h!]
    \centering
    \begin{tabular}{lllllllllll}
\toprule
   & number of buyers & \multicolumn{3}{l}{10} & \multicolumn{3}{l}{20} & \multicolumn{3}{l}{30} \\
   & beta = 1 proportion &  0.4 &  0.6 &  0.8 &  0.4 &  0.6 &  0.8 &  0.4 &  0.6 &  0.8 \\
   \hline
number of items & value distribution &      &      &      &      &      &      &      &      &      \\
\midrule
\multirow{3}{*}{40} & uniform & 0.91 & 0.88 & 0.84 & 0.92 & 0.88 & 0.86 & 0.91 & 0.93 & 0.90 \\
   & normal & 0.92 & 0.87 & 0.85 & 0.88 & 0.89 & 0.92 & 0.87 & 0.86 & 0.90 \\
   & exponential & 0.89 & 0.90 & 0.92 & 0.88 & 0.85 & 0.91 & 0.84 & 0.93 & 0.79 \\
\cline{1-11}
\multirow{3}{*}{60} & uniform & 0.88 & 0.90 & 0.89 & 0.92 & 0.86 & 0.89 & 0.92 & 0.90 & 0.90 \\
   & normal & 0.89 & 0.94 & 0.85 & 0.90 & 0.90 & 0.89 & 0.95 & 0.83 & 0.87 \\
   & exponential & 0.84 & 0.87 & 0.93 & 0.91 & 0.88 & 0.81 & 0.87 & 0.88 & 0.93 \\
\cline{1-11}
\multirow{3}{*}{80} & uniform & 0.84 & 0.92 & 0.92 & 0.87 & 0.88 & 0.92 & 0.92 & 0.95 & 0.87 \\
   & normal & 0.86 & 0.86 & 0.92 & 0.85 & 0.86 & 0.89 & 0.89 & 0.87 & 0.87 \\
   & exponential & 0.90 & 0.90 & 0.86 & 0.88 & 0.87 & 0.93 & 0.90 & 0.88 & 0.90 \\
\bottomrule
\end{tabular}

    \caption{Coverage rate of the 90\% revenue CI defined in \cref{eq:rev_variance}. Number of trials for each entry = 100.}
    \label{tbl:rev_coverage}
\end{table}

\begin{table}[h!]
    \centering
    \begin{tabular}{lllllllllll}
\toprule
    & number of buyers & \multicolumn{3}{l}{40} & \multicolumn{3}{l}{50} & \multicolumn{3}{l}{60} \\
    & beta = 1 proportion &  0.4 &  0.6 &  0.8 &  0.4 &  0.6 &  0.8 &  0.4 &  0.6 &  0.8 \\
    \hline
number of items & value distribution &      &      &      &      &      &      &      &      &      \\
\midrule
\multirow{3}{*}{100} & uniform & 0.93 & 0.89 & 0.92 & 0.85 & 0.87 & 0.87 & 0.85 & 0.85 & 0.91 \\
    & normal & 0.89 & 0.87 & 0.84 & 0.91 & 0.88 & 0.84 & 0.83 & 0.82 & 0.90 \\
    & exponential & 0.93 & 0.90 & 0.95 & 0.83 & 0.80 & 0.85 & 0.89 & 0.90 & 0.87 \\
\cline{1-11}
\multirow{3}{*}{200} & uniform & 0.82 & 0.88 & 0.86 & 0.85 & 0.88 & 0.96 & 0.88 & 0.89 & 0.81 \\
    & normal & 0.91 & 0.87 & 0.87 & 0.86 & 0.86 & 0.88 & 0.89 & 0.92 & 0.86 \\
    & exponential & 0.92 & 0.85 & 0.93 & 0.94 & 0.87 & 0.92 & 0.86 & 0.85 & 0.97 \\
\cline{1-11}
\multirow{3}{*}{400} & uniform & 0.80 & 0.92 & 0.90 & 0.91 & 0.89 & 0.92 & 0.85 & 0.89 & 0.89 \\
    & normal & 0.87 & 0.86 & 0.88 & 0.83 & 0.89 & 0.91 & 0.86 & 0.91 & 0.86 \\
    & exponential & 0.90 & 0.86 & 0.92 & 0.84 & 0.84 & 0.81 & 0.86 & 0.93 & 0.92 \\
\bottomrule
\end{tabular}

    \caption{Coverage rate of the 90\% revenue CI defined in \cref{eq:rev_variance}. Number of trials for each entry = 100.}
    \label{tbl:rev_coverage_40-60}
\end{table}

\begin{table}[h!]
    \centering
    \begin{tabular}{lllllllllll}
\toprule
    & number of buyers & \multicolumn{3}{l}{100} & \multicolumn{3}{l}{200} & \multicolumn{3}{l}{300} \\
    & beta = 1 proportion &  0.4 &  0.6 &  0.8 &  0.4 &  0.6 &  0.8 &  0.4 &  0.6 &  0.8 \\
    \hline
number of items & value distribution &      &      &      &      &      &      &      &      &      \\
\midrule
\multirow{3}{*}{100} & uniform & 0.92 & 0.88 & 0.98 & 0.90 & 0.80 & 0.92 & 0.94 & 0.92 & 0.84 \\
    & normal & 0.90 & 0.92 & 0.86 & 0.84 & 0.94 & 0.86 & 0.90 & 0.92 & 0.88 \\
    & exponential & 0.80 & 0.80 & 0.74 & 0.88 & 0.82 & 0.72 & 0.90 & 0.80 & 0.92 \\
\cline{1-11}
\multirow{3}{*}{200} & uniform & 0.82 & 0.92 & 0.90 & 0.90 & 0.90 & 0.86 & 0.92 & 0.82 & 0.90 \\
    & normal & 0.88 & 0.88 & 0.88 & 0.92 & 0.84 & 0.90 & 0.92 & 0.72 & 0.96 \\
    & exponential & 0.84 & 0.80 & 0.92 & 0.90 & 0.78 & 0.86 & 0.94 & 0.80 & 0.84 \\
\cline{1-11}
\multirow{3}{*}{400} & uniform & 0.86 & 0.90 & 0.88 & 0.98 & 0.88 & 0.94 & 0.86 & 0.86 & 0.76 \\
    & normal & 0.78 & 0.86 & 0.90 & 0.88 & 0.90 & 0.88 & 0.80 & 0.82 & 0.82 \\
    & exponential & 0.88 & 0.80 & 0.88 & 0.80 & 0.80 & 0.70 & 0.86 & 0.92 & 0.88 \\
\bottomrule
\end{tabular}

    \caption{Coverage rate of the 90\% revenue CI defined in \cref{eq:rev_variance}. Number of trials for each entry = 50.}
    \label{tbl:rev_coverage_100-300}
\end{table}

\begin{table}
    \centering
    \begin{tabular}{lllllll}
\toprule
    & treatment prob &  0.1 &  0.3 &  0.5 &  0.7 &  0.9 \\
    \hline
number of buyers & number of items &      &      &      &      &      \\
\midrule
\multirow{3}{*}{30} & 100 & 0.84 & 0.82 & 0.82 & 0.87 & 0.79 \\
    & 200 & 0.89 & 0.87 & 0.76 & 0.85 & 0.86 \\
    & 500 & 0.88 & 0.89 & 0.95 & 0.81 & 0.68 \\
\cline{1-7}
\multirow{3}{*}{60} & 100 & 0.87 & 0.88 & 0.81 & 0.87 & 0.75 \\
    & 200 & 0.89 & 0.87 & 0.85 & 0.83 & 0.85 \\
    & 500 & 0.88 & 0.91 & 0.86 & 0.85 & 0.90 \\
\cline{1-7}
\multirow{3}{*}{100} & 100 & 0.89 & 0.84 & 0.95 & 0.85 & 0.82 \\
    & 200 & 0.92 & 0.90 & 0.83 & 0.85 & 0.86 \\
    & 500 & 0.92 & 0.94 & 0.89 & 0.93 & 0.89 \\
\bottomrule
\end{tabular}

    \caption{Coverage rate of the 90\% CI for revenue treatment effect defined in \cref{eq:ab_rev_ci}. Number of trials for each entry = 100.}
    \label{tbl:abtexting} 
\end{table}

\end{document}